\documentclass[a4paper,11pt,reqno]{amsart}
\usepackage{amsmath}
\usepackage{amsthm,enumerate}

\usepackage{graphicx}
\usepackage{amssymb}

\usepackage{appendix}

\usepackage[italian,english]{babel}

\usepackage[utf8x]{inputenc}
\usepackage{fancyhdr}

\usepackage{calc}
\usepackage{url}

\usepackage[text={6in,8.6in},centering]{geometry}

\usepackage{srcltx}

\usepackage[T1]{fontenc}

\usepackage[usenames,dvipsnames]{color}

\makeatletter
\def\cleardoublepage{\clearpage\if@twoside \ifodd\c@page\else%
         \hbox{}%
     \thispagestyle{empty}
     \newpage%
     \if@twocolumn\hbox{}\newpage\fi\fi\fi}
\makeatother

\let\cleardoublepage\clearpage

\addto\captionsitalian{}

\newtheorem{thm}{Theorem}[section]
\newtheorem{cor}[thm]{Corollary}
\newtheorem{lem}[thm]{Lemma}
\newtheorem{pro}[thm]{Proposition}
\newtheorem{den}[thm]{Definition}
\newtheorem{oss}[thm]{Remark}
\numberwithin{equation}{section}

\begin{document}

\title[]{The porous medium equation with large initial data \\ on negatively curved Riemannian manifolds}

\author {Gabriele Grillo, Matteo Muratori, Fabio Punzo}

\address {Gabriele Grillo: Dipartimento di Matematica, Politecnico di Milano, Piaz\-za Leonardo da Vinci 32, 20133 Milano, Italy}
\email{gabriele.grillo@polimi.it}

\address{Matteo Muratori: Dipartimento di Matematica ``F. Casorati'', Universit\`a degli Studi di Pavia, via A. Ferrata 5, 27100 Pavia, Italy}
\email{matteo.muratori@unipv.it}

\address{Fabio Punzo: Dipartimento di Matematica e Informatica, Universit\`a della Calabria, via Pietro Bucci 31B, 87036 Rende (CS), Italy}
\email{fabio.punzo@unical.it}


%
%
%
%


\begin{abstract}
We show existence and uniqueness of very weak solutions of the Cauchy problem for the porous medium equation on Cartan-Hadamard
manifolds satisfying suitable lower bounds on Ricci curvature, with initial data that can grow at infinity at a prescribed rate, that depends crucially on the curvature bounds. The curvature conditions we require are sharp for uniqueness in the sense that if they are not satisfied then, in general, there can be infinitely many solutions of the Cauchy problem even for bounded data.
Furthermore, under matching upper bounds on sectional curvatures, we give a precise estimate for the maximal existence time, and we show that in general solutions do not exist if the initial data grow at infinity too fast. This proves in particular that the growth rate of the data we consider is optimal for existence. Pointwise blow-up is also shown for a particular class of manifolds and of initial data.
\end{abstract}
\maketitle

\tableofcontents

\section{Introduction}
We discuss existence and uniqueness of very weak solutions
of Cauchy problems for the \emph{porous medium equation} on
Riemannian manifolds, namely of the problem:
\begin{equation}\label{e64}
\begin{cases}
u_t \,=\, \Delta (u^m) & \textrm{in}\;\; M\times (0,T)\,, \\
u \,=\, u_0 & \textrm{on}\;\; M\times \{0\}\,,
\end{cases}
\end{equation}
where $M$ is an $N$-dimensional complete, simply connected Riemannian manifold with nonpositive sectional curvatures (namely a \emph{Cartan-Hadamard manifold}) and $\Delta$ is the Laplace-Beltrami
operator on $M$, $m>1$. Note that, when dealing with changing-sign solutions, as usual we set $ u^m = |u|^{m-1}u $. In particular, we are interested in considering initial data that can grow at infinity, and the interval of existence $[0, T)$ may then depend on the initial condition $u_0$.

Recently, quasilinear degenerate parabolic equations on Riemannian manifolds have attracted much attention (see e.g.~\cite{GMhyp}, \cite{GM2}, \cite{GMPrm}, \cite{MMP}, \cite{Pu1}, \cite{Pu2}, \cite{VazH}, \cite{Zhang}). In particular, in \cite{GMPrm} the well-posedness of problem \eqref{e64} with $u_0=\mu$, a finite Radon measure, has been studied; moreover, in \cite{VazH} and in \cite{GMV} smoothing estimates, support properties and the asymptotic behaviour of solutions have been addressed. The aim of our paper is to investigate existence and uniqueness of solutions to problem \eqref{e64}, considering a large class of initial conditions $u_0$, possibly unbounded at infinity. We always assume that the sectional
curvatures are nonpositive, and that the Ricci curvature is bounded from below by $-C_0(1+d(x,o)^{\gamma})$ for some constants $C_0>0,  \gamma\in(-\infty,2]$  and a fixed point $ o \in M $, where $d(\cdot,\cdot)$ denotes Riemannian distance. Some comments on the case $\gamma>2$ will be made in Remark \ref{parametri} below.

\smallskip

In the case $M=\mathbb R^N$ problem \eqref{e64} has been studied in \cite{BCP}, under optimal conditions on initial data $u_0$. In fact, in \cite{BCP} it is shown that if
\begin{equation}\label{q52}
\sup_{R\geq 1} \frac 1{R^{N+\frac 2{m-1}}} \int_{B_R} |u_0(x)|\, dx\, <\, \infty\,,
\end{equation}
then there exists a distributional solution $u$ of the differential equation in problem \eqref{e64}; in addition, for
\begin{equation*}
\alpha > \frac N 2+ \frac 1 {m-1}\,,
\end{equation*}
 $u\in C([0, T); L^1_{\alpha}(\mathbb R^N))$ and $u(0)=u_0$, where $$L^1_\alpha(\mathbb R^N):=\big\{f\in L^1_{\textrm{loc}}(\mathbb R^N)\,:\, \int_{\mathbb R^N} |f(x)| \, (1+|x|)^{-\alpha}\, dx \, <\, \infty\, \big\}\,.$$
Moreover, $u$ is the unique solution in the following sense: if $v$ is a distributional solution of problem \eqref{e64} such that, for every $\epsilon>0,\, v (1+ |x|^2)^{-\frac 1{m-1}}\in L^\infty(\mathbb R^N\times (\epsilon, T))$, then $v=u$. In addition, in \cite{BCP}, using some results from \cite{AC}, it is observed that the class of initial data that they consider is optimal for existence of solutions since initial traces must necessarily comply with \eqref{q52}.

We mention that in the proof of existence one uses the fact that if $u_0\in L^1(\mathbb R^N)\cap L^\infty(\mathbb R^N)$, then the unique weak solution $u$ of problem
\[
\begin{cases}
u_t \,=\, \Delta (u^m) & \textrm{in}\;\; \mathbb R^N\times (0,\infty)\,, \\
u \,=\, u_0 & \textrm{on}\;\; \mathbb R^N\times \{0\}\,,
\end{cases}
\]
satisfies the so-called {\it Aronson-B\'enilan} estimate (see \cite{AB})
\begin{equation}\label{q54}
\Delta(u^{m-1}) \geq - \frac {m-1}{m}\frac{N}{(m-1)N +2} \, \frac 1 t\quad \textrm{in}\;\;\, \mathfrak D'(\mathbb R^N\times (0, \infty))\,.
\end{equation}
Then, using \eqref{q54}, a certain local smoothing estimate is deduced, which is central in the proof of existence.  We are not aware of a direct analogue of such inequality in the Riemannian context till the paper \cite{LNVV}, in which some \it local \rm Aronson-B\'enilan formulas are proved by a clever use of Li-Yau type techniques (see also \cite{HHL} for some improvements). If Ricci curvature is nonnegative, a full analogue of the \it global \rm Aronson-B\'enilan inequality holds, whereas a weaker inequality holds if Ricci curvature is only bounded below. In any case, no global estimate of that kind seems available when curvature is unbounded below and when solutions are unbounded, that is the main case we shall deal with here. It is, moreover, not even clear which is the natural analogue of \eqref{q52}, since volume growth of geodesic balls should clearly appear in a condition of that type to allow for unbounded data (recall that the volume of balls can grow exponentially, or even faster, under our assumptions).

%

We are therefore forced to use a different method of proof as concerns existence. This leads us to assume \it pointwise \rm requirements on initial data, which is of course a stronger hypothesis than \eqref{q52}, but on the other hand our method is quite simple, being based on barrier arguments only, and nevertheless it singles out qualitatively the correct possible explosion rate at infinity of initial data admitting a local in time solution, as we explicitly show in Theorems \ref{opt1}, \ref{opt-blow} and in Corollary \ref{opt2}.

On the other hand, the proof of uniqueness is based on the ``duality method'' (see e.g.~\cite{ACP}, \cite{Pierre}, \cite{GMPmu}, \cite{Vaz07}). However, in order to implement such a method in \cite{BCP}  new difficulties have to be dealt with. In particular, in \cite{BCP} a crucial role is played by a supersolution to an appropriate backward parabolic problem with an unbounded coefficient; such a supersolution has the form
\begin{equation*}
\psi(x,t)=\lambda\frac{e^{\alpha(T-t)}}{(1+|x|^2)^{\beta}}\quad (x,t)\in \mathbb R^N\times [0, T]\,,
\end{equation*}
for a suitable choice of the parameters $\lambda>0, \alpha>0, \beta>0\,.$ While again  the idea of finding a supersolution will be crucial here, a separable variable solution seems not suitable to the goal. In fact, the one we shall use is taken according to the following strategy. An idea of V\'azquez, used in \cite{VazH} to deal with the PME posed on the hyperbolic space, is to rephrase the evolution for \it radial \rm (i.e. depending on the geodesic distance from some point) solutions in terms of a \it weighted, Euclidean \rm equation, in which the weight has the critical decay $|x|^{-2}$ at infinity. This strategy has been used in \cite{GMV} to deal with the PME on the class of negatively curved manifolds discussed here. It turns out that the kind of supersolution used here can be guessed from the known asymptotics of the corresponding weighted, Euclidean heat equation with critical weight as considered in \cite{IS1} (see \cite{IS2} for a generalization to the corresponding weighted PME). Such a supersolution is strictly related to the bound from below on the Ricci curvature, via the growth of the measure of the sphere as the radius increases.

\smallskip
\smallskip

Let us now go into some detail on our results. Assume that
\[
\textrm{Ric}_o(x)\geq-C_0(1+d(x,o)^{\gamma})
\]
for some constants $C_0>0, \gamma\in(-\infty,2]$, where $\textrm{Ric}_o$ denotes Ricci curvature in the radial direction associated to a point $o\in M$. Define
\begin{equation}\label{m3}
\sigma:=\frac{2-\gamma}{2}\wedge2\,;
\end{equation}
\[\rho(x):= d(x, o)\quad \textrm{for any}\;\; x\in M\,.\]
We show in Theorem \ref{thmexi} that if, for some $C>0$,
\begin{equation}\label{e21z}
|u_0(x)| \leq C \left(1+ \rho(x)\right)^{\frac{\sigma}{m-1}}\quad \textrm{for almost every } x\in M\,,
\end{equation}
then there exists a solution $u$ of problem \eqref{e64}, for some $T>0$, which satisfies an analogous bound, namely
\begin{equation}\label{e20z}
|u(x,t)| \leq C \left(1+ \rho(x)\right)^{\frac{\sigma}{m-1}}\quad \textrm{for almost every } (x,t)\in M\times (0, T)\,.
\end{equation}
Moreover, we show in Theorem \ref{thmuniq} that $u$ is the unique solution in the class of solutions satisfying condition \eqref{e20z} for some $C>0$.
Observe that both in the existence and in the uniqueness result the assumption on the Ricci curvature crucially influences the space of functions to which both the initial condition $u_0$ and the solution $u$ belong, through the parameter $\sigma$ defined in \eqref{m3}. Furthermore,  under additional upper bounds on sectional curvatures, we show in Theorem \ref{opt1} that if data have the critical growth $\rho(x)^{\frac{\sigma}{m-1}}$, the corresponding maximal existence time of solutions is at most (a multiple of) the time $T$ found in the existence theorem. In particular, if data grow at a faster rate at infinity, no positive distributional solutions exist, see Corollary \ref{opt2}.  Finally, in Theorem \ref{opt-blow} we show that on model manifolds complying with the required curvature bounds, pointwise blow-up occurs for a particular class of data which have critical growth.
This entails the sharpness of our results, in the sense that the growth condition we impose on data cannot in general be improved under the given curvature assumptions.

\smallskip

We stress that our assumption concerning the bound from below for the Ricci curvature (see \eqref{H}-(ii) below) is essential. It is not surprising that such bound on the Ricci curvature has a key role in the proof of uniqueness, since it implies \emph{stochastic completeness} of $M$, which is equivalent to uniqueness of bounded solutions in the linear case (i.e.~for the heat equation), such a condition being sharp for stochastic completeness, see \cite{IM}, \cite{Grig}. The problem of uniqueness and nonuniqueness of solutions has been the subject, in the linear setting, of extensive further research, and several sharp results have been obtained, see e.g.~\cite{M1, I, I2, M2}. In our setting, we observe that if the (negative) quadratic bound from below on the Ricci curvature is not satisfied, then the Cauchy problem \eqref{e64} with a \emph{bounded} initial datum admits infinitely many bounded solutions: this is illustrated in Remark \ref{rem: quad}.

\smallskip

The paper is organized as follows. In Section \ref{Stat} we introduce some functional analytic and geometric preliminaries; then we state the main results and we give the precise definition of solution to problem
\eqref{e64}. Existence of solutions, and preliminarily a key a priori estimate are shown in Section \ref{existence}. In Section \ref{uniqueness} we prove
uniqueness of solutions. Finally, in Section \ref{blup} we prove the blow-up result and the nonexistence theorem, thus showing the sharpness of our results.

\section{Preliminaries, assumptions and statements of the main results}\label{Stat}
We collect here notations concerning the geometric objects we deal with, well-known Laplacian comparison results used in the sequel, and the corresponding geometric assumptions which are supposed to hold throughout the paper. We shall also recall some definitions and preliminary results on the function spaces necessary to our discussion. Finally we shall state our results, first as concerns existence and uniqueness, and then as concerns maximal existence time, nonexistence and blow-up for suitable classes of data.

\subsection{Notations from Riemannian geometry}\label{RG}
Let $M$ be a complete noncompact Riemannian manifold. Let $\Delta$
denote the Laplace-Beltrami operator, $\nabla$ the Riemannian
gradient and $d\mu$ the Riemannian volume element on $M$.


We consider \emph{Cartan-Hadamard} manifolds, i.e.~complete, noncompact, simply connected Riemannian manifolds with nonpositive
sectional curvatures everywhere. Observe that on Cartan-Hadamard manifolds the \emph{cut locus} of any point $o$ is empty \cite{Grig,Grig3}.
Hence, for any $x\in M\setminus \{o\}$ one can define its {\it polar coordinates} with pole at $o$, namely $\rho(x) := d(x, o)$ and $\theta\in \mathbb S^{N-1}$. If we denote by $B_R$ the Riemannian ball of radius $R$ centred at $o$ and $ S_R:=\partial B_R $, there holds
\begin{equation}\label{n51}
\textrm{meas}(S_R)\,=\, \int_{\mathbb S^{N-1}}A(\rho, \theta) \, d\theta^1d \theta^2 \ldots d\theta^{N-1}\,,
\end{equation}
for a specific positive function $A$ which is related to the metric tensor, \cite[Sect. 3]{Grig}. Moreover, it is direct to see that the Laplace-Beltrami operator in polar coordinates has the form
\begin{equation}\label{e1}
\Delta \,=\, \frac{\partial^2}{\partial \rho^2} + m(\rho, \theta) \, \frac{\partial}{\partial \rho} + \Delta_{S_{\rho}} \, ,
\end{equation}
where $m(\rho, \theta):=\frac{\partial }{\partial \rho}(\log A)$ and $ \Delta_{S_{\rho}} $ is the Laplace-Beltrami operator on $ S_{\rho} $.  Thanks to \eqref{e1}, we can identify $ m(\rho,\theta) $ as the Laplacian of the distance function $ x \mapsto \rho(x) $.

Let $$\mathcal A:=\left\{f\in C^\infty((0,\infty))\cap C^1([0,\infty)): \, f'(0)=1, \, f(0)=0, \, f>0 \ \textrm{in}\;\, (0,\infty)\right\} .$$ We say that $M$ is a {\it spherically symmetric manifold} or a {\it model manifold} if the Riemannian metric is given by
\begin{equation*}\label{e2}
ds^2 \,=\, d\rho^2+\psi(\rho)^2 \, d\theta^2,
\end{equation*}
where $d\theta^2$ is the standard metric on $\mathbb S^{N-1}$ and $\psi\in \mathcal A$. In this case, we shall write $M\equiv M_\psi$; furthermore, we have $A(\rho,\theta)=\psi(\rho)^{N-1} \, \eta(\theta) $ for a suitable angular function $\eta$, so that
\begin{equation}\label{mm43}
\Delta \,=\, \frac{\partial^2}{\partial \rho^2} + (N-1) \, \frac{\psi'}{\psi} \, \frac{\partial}{\partial\rho} + \frac1{\psi^2} \, \Delta_{\mathbb S^{N-1}} \, .
\end{equation}
Note that $\psi(r)=r$ corresponds to $M=\mathbb R^N$, while $\psi(r)=\sinh r$ corresponds to $ M=\mathbb H^N $, namely the $N$-dimensional hyperbolic space.

For any $x\in M\setminus\{o\}$, we denote by $\textrm{Ric}_o(x)$ the
\emph{Ricci curvature} at $x$ in the radial direction
$\frac{\partial}{\partial\rho}$. Let $\omega$ be any pair of tangent
vectors from $T_x M$ having the form $\big(\frac{\partial}{\partial \rho}
, V \big)$, where $V$ is a unit vector orthogonal to
$\frac{\partial}{\partial\rho}$. We denote by $\textrm{K}_{\omega}(x)$ the \emph{sectional curvature} at $x\in M$ of the $2$-section determined by $\omega$.

\subsection{Laplacian comparison}\label{lapl-comp}
Let us recall some crucial Laplacian comparison results. It is by now classical (see e.g.~\cite{GW} and \cite[Section
15]{Grig}) that if
\begin{equation}\label{e3a}
\textrm{K}_{\omega}(x)\leq -\frac{\psi''(\rho)}{\psi(\rho)}\quad \textrm{for all } x \equiv (\rho,\theta)\in M\setminus\{o\}
\end{equation}
for some function $\psi\in \mathcal A$, then
\begin{equation}\label{e3}
m(\rho, \theta)\geq (N-1) \, \frac{\psi'(\rho)}{\psi(\rho)} \quad \textrm{for all } \rho>0 \, , \ \theta \in \mathbb S^{N-1}\,.
\end{equation}
On the other hand, if
\begin{equation}\label{e3c}
\textrm{Ric}_{o}(x) \geq -(N-1) \, \frac{\psi''(\rho)}{\psi(\rho)} \quad \textrm{for all } x\equiv(\rho,\theta)\in M\setminus\{o\}
\end{equation}
for some function $\psi\in \mathcal A$, then
\begin{equation}\label{e4}
m(\rho, \theta)\leq (N-1) \, \frac{\psi'(\rho)}{ \psi(\rho)}\quad \textrm{for all } \rho>0 \, , \ \theta \in \mathbb S^{N-1}\,.
\end{equation}
Furthermore, in the former case from \eqref{n51} and \eqref{e3} it follows that
\begin{equation*}\label{n50z}
\textrm{meas}(S_R)\geq C \, \psi(R)^{N-1} \quad \forall R>0 \, ,
\end{equation*}
whereas in the latter case from \eqref{n51} and \eqref{e4} it follows that
\begin{equation}\label{n50a}
\textrm{meas}(S_R)\leq C \, \psi(R)^{N-1} \quad \forall R>0 \, ,
\end{equation}
for some constant $C>0$ independent of $ R $.

In the special case of a model manifold $M_\psi$, for any $x\equiv(\rho, \theta)\in M_\psi\setminus\{o\}$ we have
\begin{equation}\label{e1ce}
\textrm{K}_{\omega}(x)=-\frac{\psi''(\rho)}{\psi(\rho)} \, , \quad \textrm{Ric}_{o}(x)=-(N-1) \, \frac{\psi''(\rho)}{\psi(\rho)} \, .
\end{equation}
Moreover, the sectional curvature w.r.t.~planes orthogonal to $\frac{\partial}{\partial \rho}$ is given by
\begin{equation*}
\frac{1-[\psi'(\rho)]^2}{\psi(\rho)^2} \, .
\end{equation*}
In particular, as $\psi\in \mathcal A$, the condition $ \psi''\geq 0 $ in $(0, \infty)$ is necessary and sufficient for $ M_\psi $ to be a Cartan-Hadamard manifold.

Finally, note that for any Cartan-Hadamard manifold we have $\textrm{K}_{\omega}(x)\leq 0$, so that \eqref{e3a} is trivially satisfied with
$\psi(\rho)=\rho$ and therefore
\begin{equation}\label{e5}
m(\rho, \theta) \geq \frac{N-1}{\rho} \quad \textrm{for any } x \equiv (\rho, \theta) \in M \setminus \{o\}\,.
\end{equation}

\subsection{Main assumptions and consequences}\label{main}
Throughout the paper we shall work under the following hypotheses:
\begin{equation} \tag{{\it H}} \label{H}
\begin{cases}
\textrm{(i)} & M \ \textrm{is a Cartan-Hadamard manifold of dimension $N\ge2$} \, ; \\
\textrm{(ii)} & \textrm{Ric}_o(x)\geq -C_0 \left(1+d(x,o)^\gamma \right) \ \textrm{for some } C_0>0 \ \textrm{and } \gamma\in(-\infty,2] \, .
\end{cases}
\end{equation}
For instance, assumption \eqref{H} is satisfied if $M=\mathbb H^N$, with $ \gamma=0 $. More generally, it is not difficult to show that \eqref{H} is met e.g.~by Riemannian models associated with suitable convex functions $\psi$ such that
\begin{equation}\label{e50z}
\psi(\rho) = e^{f(\rho)} \, , \quad f(\rho) \sim C \, \rho^{1+\frac{\gamma}2} = C \, \rho^{2-\sigma} \quad \textrm{as } \rho \to \infty \quad \textrm{if } \gamma\in(-2, 2]
\end{equation}
or
\begin{equation}\label{e51z}
\psi(\rho) \sim C \, \rho^{\delta} \quad \textrm{as } \rho \to \infty \, , \quad \delta := \frac{1+\sqrt{1+ 4C_0/(N-1)}}{2} \, , \quad \textrm{if } \gamma=-2 \, ,
\end{equation}
where $C$ are positive constants. If $\gamma<-2$ the corresponding models are very close to the Euclidean space, namely $\psi(\rho) \sim C \, \rho $ as $\rho\to\infty$. Here by $ f(\rho) \sim g(\rho) $ we mean that the ratio $ f(\rho)/g(\rho) $ tends to $1$ as $ \rho \to \infty $.

We refer the reader to \cite[Section 2.3]{GMV} for more details in this regard. Below, when it is needed, we shall be more precise and show how it is possible, under our curvature assumptions, to exploit the Laplacian comparison results recalled above by using specific model manifolds whose behaviour at infinity is indeed the same as \eqref{e50z}, \eqref{e51z} or $ \rho $: see Lemma \ref{prop-comp-ricci} and Lemma \ref{prop-comp-sect}.

\subsection{Functional setting}\label{sf}
In the sequel we shall consistently make use of the functional space $ X_{\infty,\sigma} $, which is defined as the space of all functions $ f \in L^\infty_{\rm loc}(M) $ such that
\begin{equation*}\label{e40z}
|f(x)| \leq C \left(1+ \rho(x)\right)^{\frac{\sigma}{m-1}} \quad \textrm{for a.e. } x \in M
\end{equation*}
for some $ C>0 $, which in general depends on $f$. For later purposes, for all $ r \ge 1 $ we endow $ X_{\infty,\sigma} $ with the norms
\begin{equation}\label{e40-norm}
\left\| f \right\|_{\infty,r} := \sup_{x \in M} \frac{\left| f(x) \right|}{\left(r^2+ \rho(x)^2\right)^{\frac{\sigma}{2(m-1)}}} \, .
\end{equation}
Note that, by definition,
\begin{equation}\label{e40-bis}
|f(x)| \leq \left\| f \right\|_{\infty,r} \left(r^2 + \rho(x)^2 \right)^{\frac{\sigma}{2(m-1)}} \quad \textrm{for a.e. } x \in M \, .
\end{equation}
Moreover, it is readily checked that $ r \mapsto \| f \|_{\infty,r} $ is nonincreasing and satisfies
\begin{equation}\label{e40-limsup-final}
\lim_{r \to \infty}  \left\| f \right\|_{\infty,r} = \limsup_{\rho(x)\to\infty} \frac{\left| f(x) \right|}{\rho(x)^{\frac{\sigma}{m-1}}} \, ,
\end{equation}
a useful identity that we shall exploit below.

\subsection{Existence and uniqueness results}\label{sect: exuni}

In this section we first provide a general notion of solution to \eqref{e64} and then establish well-posedness results for initial data and solutions belonging to $X_{\infty,\sigma}$.

\begin{den}\label{defsol}
Let $T>0$ and $ u_0 \in L^\infty_{\rm{loc}}(M) $. We say that $u\in L^\infty_{\rm{loc}}(M\times [0, T))$ is a (very weak) solution of problem \eqref{e64} in the time interval $[0,T)$
if
\begin{equation}\label{q50}
-\int_0^T \int_M u\, \phi_t\,  d\mu dt \,=\, \int_0^T \int_M u^m \, \Delta \phi\, d\mu dt + \int_M u_0(x) \, \phi(x,0)\, d\mu
\end{equation}
for all $\phi\in C^\infty_c(M\times [0, T))$.

Furthermore, we say that $u\in L^\infty_{\rm{loc}}(M\times [0, T))$ is a subsolution (supersolution) of problem \eqref{e64} if \eqref{q50} holds with ``$=$'' replaced by ``$\leq$'' (``$\geq$''), for all $\phi\in C^\infty_c(M\times [0, T))$ with $ \phi\geq 0$.
\end{den}

Solutions will from now on be understood in the very weak sense described above, and we shall often refer to them simply as ``solutions'' since no ambiguity occurs.

\smallskip

Concerning existence, we have the following result.
\begin{thm}[Existence]\label{thmexi}
Let assumption \eqref{H} be satisfied with $\gamma\in(-\infty,2)$. Let $u_0$ be a measurable function satisfying
\[
|u_0(x)| \leq C \left( 1 + \rho(x) \right)^{\frac{\sigma}{m-1}} \quad \textrm{for a.e. } x\in M
\]
for some $C>0$, with $\sigma$ being given in \eqref{m3}. Then there exists a solution $u$
of problem \eqref{e64} with
\begin{equation}\label{n2}
T= \underline{C} \left\| u_0 \right\|^{1-m}_{\infty,r} ,
\end{equation}
%
where $\underline{C}$ is a positive constant depending on $C_0,\gamma,N, m$ but not on $  r \ge 1 $. Furthermore, $u$ satisfies the pointwise estimate
\begin{equation}\label{thm-limit-abs}
\left| u(x,t) \right| \le \left( 1 - \frac{t}{T} \right)^{-\frac{1}{m-1}} \left\| u_0 \right\|_{\infty,r} \left( r^2+ \rho(x)^2 \right)^{\frac{\sigma}{2(m-1)}} \quad \textrm{for a.e. } (x,t) \in M \times (0,T) \, .
\end{equation}
In particular,
\begin{equation}\label{thm-norms}
\left\| u(t) \right\|_{\infty,r} \le \left( 1 - \frac{t}{T} \right)^{-\frac{1}{m-1}} \left\| u_0 \right\|_{\infty,r} \quad \textrm{for a.e. } t \in (0,T)
\end{equation}
and
\begin{equation}\label{thm-limsup}
\limsup_{\rho(x)\to\infty} \frac{\left| u(x,t) \right|}{\rho(x)^{\frac{\sigma}{m-1}}} \le \left( 1 - \frac{t}{T} \right)^{-\frac{1}{m-1}} \, \limsup_{\rho(x)\to\infty} \frac{\left| u_0(x) \right|}{\rho(x)^{\frac{\sigma}{m-1}}} \quad \textrm{for a.e. } t \in (0,T)  \, .
\end{equation}
As a consequence, the solution exists at least up to
\begin{equation}\label{thm-tmax}
T=\underline{C} \left[ \limsup_{\rho(x)\to\infty} \frac{\left| u_0(x) \right|}{\rho(x)^{\frac{\sigma}{m-1}}} \right]^{1-m} ,
\end{equation}
so that $ \lim_{\rho(x)\to\infty} \rho(x)^{-\frac{\sigma}{m-1}} \, | u_0(x) | = 0 $ implies global existence.
\end{thm}

\smallskip

In the class of solutions belonging to $ X_{\infty,\sigma} $, we can also establish uniqueness.
\begin{thm}[Uniqueness]\label{thmuniq}
Let assumption \eqref{H} be satisfied. Let $u, v$ be any two solutions of problem \eqref{e64} corresponding to the same $ u_0 \in X_{\infty,\sigma} $, up to the same time $T>0$. Suppose that \eqref{e20z} holds both for $u$ and $v$. Then $u = v$ a.e.~in $M\times (0, T)$.
\end{thm}

Note that both in the existence and in the uniqueness result the curvature assumption in \eqref{H} is involved via the constant $\sigma$ defined by \eqref{m3}, which in turn affects the spaces of functions to which the initial data and the solutions belong.
%
%
%
%

\begin{oss}\label{rem: quad} \rm
If assumption $(H)$ is not satisfied, in general, the solution of problem \eqref{e64} with $u_0\in L^\infty(M)$ (which always exists, see Remark \ref{parametri}) is not unique in $L^\infty(M\times (0,T)).$ In fact, let $M\equiv M_\psi$ be any model manifold with nonpositive sectional curvature such that
\begin{equation}\label{e3ce}
\operatorname{Ric}_o(x)\leq - C_0 \, \rho^{2+\epsilon} \quad \forall x \equiv (\rho, \theta) \in M : \ \rho>1 \, , \quad \textrm{for some} \ \epsilon>0 \ \textrm{and} \ C_0>0 \, .
\end{equation}
Clearly hypothesis $(H)$-(ii) is not met for all $ \epsilon>0 $. The curvature condition \eqref{e3ce} is associated with model manifolds whose model function $ \psi(\rho) $ behaves like $ e^{ f(\rho)}$ with $ f(\rho) \sim C \, \rho^{2+\frac{\epsilon}{2}}$ as $\rho \to \infty$, for a suitable $ C>0 $. In view of \eqref{e1ce}, we also have that $(N-1) \operatorname{K}_{\omega}(x)\leq - C_0 \, \rho^{2+\epsilon} $ for all $ x\equiv (\rho, \theta)\in M$ such that $ \rho>1$. This ensures that the hypotheses of \cite[Theorem 15.4(b)]{Grig} are satisfied; therefore, as it is observed in the proof of that theorem, we have that
\begin{equation}\label{e2ce}
\int_{M_\psi} \mathcal{G}(o,x) \, d\mu(x) = \int_0^\infty \frac{\mathcal{V}(r)}{\mathcal{S}(r)} \, dr <\infty \, ,
\end{equation}
where $ \mathcal{G} $ is the Green function of $ M_\psi $, $\mathcal{S}(r):=\textrm{meas}(S_r)$ and $ \mathcal{V}(r):=\int_0^r \mathcal{S}(\xi) d\xi,$ the latter being the volume of the ball $B_r$ of radius $ r>0 $. Actually in \cite[Theorem 15.4(b)]{Grig} \eqref{e2ce} is shown to hold for a model manifold satisfying \eqref{e3ce} with equality. Nevertheless, standard comparison theorems for heat kernels (and therefore Green functions, see e.g.~\cite[Theorem 4.2]{Grig2} and \cite[Section 3]{GMPrm}) ensure that, as a consequence, \eqref{e2ce} also holds for $ M_\psi $. Thanks to \eqref{e2ce} we can then apply \cite[Remark 3.12 and Theorem 3.8]{PuAA} to infer that for any  $u_0\in L^\infty(M)$, $ u_0\geq 0$, problem \eqref{e64} admits infinitely many bounded solutions in the sense of Definition \ref{defsol}. More precisely, for every $u_0\in L^\infty(M)$, $ u_0\geq 0$, and for every Lipschitz function $A$ defined in $[0, T]$ with $A(0)=0$, $ A'\geq 0$, there exists a nonnegative bounded solution of problem \eqref{e64} such that $\int_0^t u^m(x,s) ds \to A(t)$ as $\rho(x)\to \infty$ uniformly w.r.t.~$t\in [0, T]$. Hence, in particular, nonuniqueness for problem \eqref{e64} occurs. In the linear case $m=1$, these results are in agreement with the nonuniqueness issues entailed by \cite[Theorem 6.2(4) and Corollary 15.2(d)]{Grig}.
\end{oss}

\subsection{Maximal existence time, nonexistence and blow-up results}
It is worth noting that Theorem \ref{thmexi} does not provide the maximal existence time for a solution to \eqref{e64} corresponding to an initial datum in $ X_{\infty,\sigma} $, but only a lower estimate for the latter given by \eqref{n2}. In principle the solution, in the sense of Definition \ref{defsol}, could even exist as such for \emph{all} times (which is indeed the case if e.g.~$ u_0 \in L^\infty(M) $).

We stress that, in the next results, by ``maximal existence time'' not only we mean the largest time up to which the solution exists in $X_{\infty,\sigma}$, but more in general the largest time up to which \emph{any nonnegative solution}, in the sense of Definition \ref{defsol}, can exist.

Under an additional upper bound on the sectional curvature that matches the assumed lower bound on the Ricci curvature,
we can show that initial data with a prescribed power-type growth at infinity do give rise to solutions that cease to exist in finite time.
More precisely, we have the following.

\begin{thm}[Maximal existence time]\label{opt1}
Let assumption \eqref{H} be satisfied for some  $\gamma \in (-\infty, 2)$. If $ \gamma \in (-2,2) $, assume in addition that there exist $ C_1 , R_1 > 0 $ such that
\begin{equation}\label{assumption-sectional-intro}
\operatorname{K}_{\omega}(x)\leq - C_1 \, d(x,o)^\gamma \quad \forall x \in M \setminus B_{R_1}  \, .
\end{equation}
Let $ u_0 \in X_{\infty,\sigma} $ be any nonnegative initial datum satisfying
\begin{equation}\label{tmax-opt}
\liminf_{\rho(x)\to\infty} \frac{u_0(x)}{\rho(x)^{\frac{\sigma}{m-1}}} > 0 \, .
\end{equation}
Then there exists a positive constant $ \overline{C} $, depending only on $C_1, R_1,\gamma,N,m$, such that the maximal time $T$ for which the corresponding solution $u$ of problem \eqref{e64} exists, satisfies
\begin{equation}\label{mm30-datum}
T \le \overline{C} \left[ \liminf_{\rho(x)\to\infty} \frac{u_0(x)}{\rho(x)^{\frac{\sigma}{m-1}}} \right]^{1-m} .
\end{equation}
In particular, for some $ 0 < \tau \le T  $ there holds
\begin{equation}\label{mm30}
\limsup_{t\to \tau^-} \, \| u(t) \|_{\infty,r}  = \infty
\end{equation}
for all $ r \ge 1 $.
\end{thm}

As a direct consequence of Theorem \ref{opt1}, we also have a nonexistence result for initial data growing at infinity \emph{faster} than the critical power $\rho(x)^{\frac{\sigma}{m-1}}$.
\begin{cor}[Nonexistence]\label{opt2}
Let the assumptions of Theorem \ref{opt1} be fulfilled. Let $ u_0 \in L^\infty_{\rm loc}(M) $ be any nonnegative initial datum satisfying
\begin{equation}\label{mm22}
\lim_{\rho(x)\to \infty} \frac{u_0(x)}{\rho(x)^{\frac{\sigma}{m-1}}} = +\infty \, .
\end{equation}
Then problem \eqref{e64} does not admit any nonnegative solution in $[0,T)$, given any $T>0$.
\end{cor}

In the case of model manifolds $ M \equiv M_\psi $, given any $T>0$ we can provide classes of solutions which \emph{blow up pointwise} at the maximal existence time $T$. To this end, a relevant role is played by the positive solutions $ \mathsf{W}_{T,\alpha} $ (let $ T,\alpha>0 $) to the following Cauchy problem:
\begin{equation}\label{pb: cauchy}
\begin{cases}
\frac{1}{(m-1)T} \, \mathsf{W}_{T,\alpha} = \big(\mathsf{W}^m_{T,\alpha} \big)^{\prime\prime} + (N-1) \, \frac{\psi^\prime}{\psi} \big(\mathsf{W}^m_{T,\alpha} \big)^{\prime} \quad \textrm{in } (0,\infty) \, , \\
\mathsf{W}_{T,\alpha}^\prime(0) = 0 \, , \\
\mathsf{W}_{T,\alpha}(0) = \alpha \, .
\end{cases}
\end{equation}
Under suitable curvature assumptions on $ M_\psi $, the well-posedness of \eqref{pb: cauchy}, as well as the fact that solutions belong to $ X_{\infty,\sigma} $ and are ordered w.r.t.~$ \alpha $, will be addressed in detail in the end of Section \ref{blup}.

\begin{thm}[Pointwise blow-up]\label{opt-blow}
Let $ M \equiv M_{\psi} $ be any model manifold satisfying hypothesis \eqref{H} for some $ \gamma \in (-\infty,2) $. Let $T>0$ be fixed but arbitrary. If $ \gamma \in (-2,2) $, assume in addition that \eqref{assumption-sectional-intro} holds for some $ C_1 , R_1 > 0 $. Let $ u_0 $ be any initial datum complying with
\begin{equation}\label{W-W}
\mathsf{W}_{T, \alpha_0}(\rho(x)) \le u_0(x) \le \mathsf{W}_{T,\alpha_1}(\rho(x)) \quad \textrm{for a.e. } x \in M
\end{equation}
for some $ \alpha_1>\alpha_0>0 $. Then the maximal existence time for the corresponding solution $u$ is exactly $T$, where
\begin{equation}\label{eq: u0-blowup}
\frac{k_0}{T^{\frac{1}{m-1}}} \le \liminf_{\rho(x)\to\infty} \frac{ u_0(x) }{\rho(x)^{\frac{\sigma}{m-1}}} \le \limsup_{\rho(x)\to\infty} \frac{ u_0(x) }{\rho(x)^{\frac{\sigma}{m-1}}} \le \frac{k_1}{T^{\frac{1}{m-1}}}
\end{equation}
for positive constants $k_0$ and $k_1$ depending on $ C_0,\gamma,N,m $ and on $ C_1, R_1,\gamma,N,m $, respectively, but not on $\alpha$. More precisely, $u$ blows up at $ t=T $ almost everywhere, that is
\begin{equation}\label{blow-up-ptws}
\lim_{t \to T^-} u(x,t) = +\infty \quad \textrm{for a.e. } x \in M \, .
\end{equation}
\end{thm}

Note that \eqref{W-W} is essentially a condition \it at infinity \rm only (plus a positivity requirement).

\smallskip

We point out that when $ \gamma \in (-2,2) $, in order to prove Theorems \ref{opt1}, \ref{opt-blow} and Corollary \ref{opt2}, hypothesis \eqref{assumption-sectional-intro} on sectional curvatures is essential. Indeed, it is plain that the Euclidean space $ \mathbb{R}^N $ does fulfil assumption \eqref{H} for any such $ \gamma $: however, in this case initial data growing like $ \rho(x)^{2/(m-1)} $ are allowed for short-time existence, as well as data with slower growth (in particular like $ \rho(x)^{\sigma/(m-1)} $) guarantee \emph{global} existence.

\begin{oss}\label{parametri}\rm
As the reader may note, in Theorems \ref{thmexi}, \ref{opt1}, \ref{opt-blow} and Corollary \ref{opt2} we do not address the critical case $\gamma=2$. Actually, since $ \sigma=0 $, with such a choice we would have $X_{\infty, \sigma} \equiv L^\infty(M)$. On the other hand, by standard methods (see {e.g.}~\cite{Vaz07}), it can be easily proved that problem \eqref{e64} admits global bounded solutions on any Cartan-Hadamard manifold. Moreover, as a consequence of Theorem \ref{thmuniq}, uniqueness of bounded solutions holds up to $ \gamma=2 $. Yet it remains to be understood whether growing data can still be allowed (clearly, not with a power rate).
Accordingly, in the case $\gamma>2$ one should possibly investigate well-posedness results for initial data suitably \emph{vanishing} at infinity: in this regard recall Remark \ref{rem: quad}.
\end{oss}

\section{Existence: proofs}\label{existence}

This section is devoted to establishing Theorem \ref{thmexi}. To begin with, we outline the main ideas behind our method of proof.

\subsection{Outline of the strategy}
Suppose first that $ u_0 \in X_{\infty, \sigma}$ and $u_0 \geq 0$. For every $R>0$, let us consider the approximate problems
\begin{equation}\label{q10}
\begin{cases}
u_t = \Delta \! \left(u^m\right) & \textrm{in } B_R\times (0,\infty)\,, \\
u = 0 & \textrm{on } \partial B_R \times (0, \infty)\,, \\
u = u_0 \rfloor_{B_R} & \textrm{on } B_R\times \{0\}\,.
\end{cases}
\end{equation}
Existence and uniqueness of a weak solution $u_R$ of problem \eqref{q10}, in the sense of Definition \ref{def2} below, can easily be obtained by standard methods (see Proposition \ref{pro1}). Moreover, thanks to the comparison principle, $ u_R $ is nondecreasing w.r.t.~$R$. Hence, there exists the pointwise limit $u:=\lim_{R\to \infty} u_R$ in $M \times (0,\infty) $. In order to pass to the limit in \eqref{q10} (with $ u \equiv u_R $) so as to show that $u$ solves \eqref{e64}, we need some a priori bound over $ u_R $, which at least guarantees local boundedness of $u$. This will be provided by an explicit separable supersolution $ \overline{u} $ of \eqref{e64}, which exists in $ M \times (0,T) $ and blows up at $ t=T $, the latter being a positive time related to $ u_0 $. Another application of the comparison principle on balls then ensures the validity of the crucial estimate
\begin{equation}\label{m71}
u_R(x,t) \leq \overline{u}(x,t) \quad \textrm{for a.e. } (x,t) \in M \times (0,T) \, , \ \forall R>0\,.
\end{equation}
In the case $ u_0 $ does not have a sign, it is possible to establish also a lower bound analogous to \eqref{m71}. However, the sequence $ u_R $ is in general not monotone, and the passage to the limit in \eqref{q10} has to be carried out by means of weak$^\ast$-convergence arguments. See Section \ref{sect: proofExi} for the details.

\smallskip
We recall here the standard notion of (weak) solution to \eqref{q10}.
\begin{den}\label{def2}
Let $u_0\in L^\infty(B_R)$. By a weak solution of problem \eqref{q10} we mean a function $u\in C([0,\infty); L^1(B_R))\cap L^\infty(B_R\times (0, \infty))$ such that $u^m\in L^2_{\textnormal{loc}}((0,\infty); H^1_0(B_R))$, which satisfies
\[
\int_0^\infty \int_{B_R} u \, \phi_t \, d\mu dt \, = \, \int_0^\infty \int_{B_R} \left\langle \nabla u^m , \nabla \phi \right\rangle d\mu dt
\]
for any $\phi\in C^\infty_c(B_R\times (0, \infty))$ and $u(0) = u_0$.
\end{den}

The following well-posedness result holds for weak solutions to \eqref{q10}, which can be proved exactly as in the case $ M \equiv \mathbb{R}^N $ (see \cite[Section 5]{Vaz07}).
\begin{pro}\label{pro1}
Let $u_0\in L^\infty(B_R)$. Then there exists a unique weak solution $u$ of problem \eqref{q10}, in the sense of Definition \ref{def2}. 
\end{pro}

\subsection{Construction of the supersolution}
In this section we provide the supersolution $ \overline{u} $ claimed above, whose existence is fundamental in order to prove our existence theorem. To this end, we first need to establish an auxiliary result involving the Laplacian of suitable functions on the kind of manifolds we are interested in.
\begin{lem}\label{prop-comp-ricci}
Let assumption \eqref{H} be satisfied. Then there exists a positive constant $ C^\prime $, depending on $ C_0,\gamma,N $, such that
\begin{equation}\label{eq: m-est-above}
m(\rho,\theta) \le \frac{C^\prime}{\rho\left( 1+\rho \right)^{\sigma-2}} \quad \forall x \equiv (\rho,\theta) \in M \setminus \{ o \} \, .
\end{equation}
\end{lem}
\begin{proof}
We give a sketchy proof for the reader's convenience, since this result in fact had already been established in \cite{GMV}. The idea is elementary and relies on the Laplacian comparison inequalities recalled in Section \ref{lapl-comp}: it suffices to provide a function $ \psi \in \mathcal{A} $ which complies with
\begin{equation}\label{eq: m-est-above-diffeq}
C_0 \left(1+\rho^\gamma \right) \leq (N-1) \, \frac{\psi''(\rho)}{\psi(\rho)} \quad \forall \rho>0
\end{equation}
and satisfies
\begin{equation}\label{eq: m-est-above-diffeq-2}
(N-1) \, \frac{\psi'(\rho)}{\psi(\rho)} \le \frac{C^\prime}{\rho\left( 1+\rho \right)^{\sigma-2}} \quad \forall \rho>0 \, .
\end{equation}
In general, the differential equation associated with \eqref{eq: m-est-above-diffeq} does not admit explicit solutions. Nevertheless, by means of standard ODE techniques, it is not difficult to show that there exists indeed some $ \psi \in \mathcal{A} $ fulfilling \eqref{eq: m-est-above-diffeq}--\eqref{eq: m-est-above-diffeq-2}. However, it is necessary to distinguish carefully between the cases $ \gamma \in (-2,2] $, $ \gamma=-2 $ and $ \gamma < -2 $. In the first case, one can proceed as in \cite[Lemma 4.2]{GMV}: it is also possible to show that the associated $ \psi $ behaves like \eqref{e50z}. In the second case we refer to \cite[Section 8.1]{GMV}: the associated $ \psi $ behaves like \eqref{e51z}. Finally, in the third case one argue as in \cite[Section 8.2]{GMV}, and the associated $ \psi $ behaves like $ \rho $ up to multiplicative constants.
\end{proof}

We are now ready to exhibit the required supersolution.

\begin{pro}\label{supersol} Let assumption \eqref{H} be satisfied with $\gamma \in (-\infty,2)$. Given any $T,a>0$ and $ r \ge 1 $, set
\begin{equation*}\label{t1c}
\overline{u}(x,t):=\left( 1 - \frac{t}{T} \right)^{-\frac{1}{m-1}} W_{T,r}(x) \quad \forall (x,t) \in M \times [0,T) \, ,
\end{equation*}
where
\begin{equation}\label{eq: expl-subsol}
W_{T,r}(x) := \frac{a \left( r^2+ \rho(x)^2 \right)^{\frac{\sigma}{2(m-1)}}}{T^{\frac{1}{m-1}}} \quad \forall x \in M \, .
\end{equation}
Then there exists $ a=a(C_0,\gamma,N,m)>0 $ such that
\begin{equation}\label{t1a}
\overline{u}_t \geq \Delta \, \overline{u}^m \quad \textrm{in } M\times (0, T) \, .
\end{equation}
\end{pro}
\begin{proof}
Let
\begin{equation}\label{mm18}
W(x)\equiv W(\rho(x)):= a \left( r^2+ \rho(x)^2 \right)^{\frac{\sigma}{2(m-1)}}\quad \forall x\equiv (\rho, \theta)\in M \, .
\end{equation}
We have:
\begin{equation*}
\left(W^m\right)^\prime(\rho)=a^m \, \frac{\sigma m}{m-1} \, \rho \left(r^2 + \rho^2\right)^{\frac{\sigma m}{2(m-1)}-1} \, ,
\end{equation*}
\begin{equation}\label{mm7}
\left(W^m\right)''(\rho) = a^m \, \frac{\sigma m}{m-1} \left( r^2 + \rho^2 \right)^{\frac{\sigma m}{2(m-1)}-1} + a^m \, \frac{\sigma m}{m-1} \left(\frac{\sigma m}{m-1} -2 \right) \rho^2 \left(r^2 + \rho^2\right)^{\frac{\sigma m}{2(m-1)}-2} \, .
\end{equation}
We claim that we can select $a>0$, only depending on $ \sigma$, $m $ and on the constant $ C^\prime $ of Lemma \ref{prop-comp-ricci}, in such a way that
\begin{equation}\label{t1}
W \geq (m-1) \, \Delta\left(W^m\right) \quad \textrm{in } M \, .
\end{equation}
Indeed, in view of \eqref{e1} and \eqref{mm7}, we can infer that
\begin{equation*}\label{t1b}
(m-1)\,\Delta\left(W^m\right) = a^m \, \sigma m \left(r^2+\rho^2\right)^{\frac{\sigma m}{2(m-1)}-1}\left[ 1 + \left(\frac{\sigma m}{m-1} -2 \right) \frac{\rho^2}{r^2+\rho^2} + \rho \, m(\rho, \theta) \right] .
\end{equation*}
In particular, due to \eqref{eq: m-est-above}, we have that
\begin{equation}\label{t2}
(m-1) \, \Delta\left(W^m\right) \leq a^m \, \sigma m \left(r^2+\rho(x)^2\right)^{\frac{\sigma m}{2(m-1)}-1}\left(1+ \left|\frac{\sigma m}{m-1} -2\right| + 2^{2-\sigma} \, C' \right) \quad \textrm{in } B_1
\end{equation}
and
\begin{equation}\label{t3}
(m-1) \, \Delta\left(W^m\right) \leq a^m \, \sigma m \left(r^2+\rho(x)^2\right)^{\frac{\sigma m}{2(m-1)}-1} \! \left[ 1+ \left|\frac{\sigma m}{m-1} -2\right| + C' \left(1+\rho(x)\right)^{2-\sigma} \right] \ \ \textrm{in } B_1^c  \, .
\end{equation}
Thanks to \eqref{t2}--\eqref{t3} and to the explicit expression \eqref{mm18} of $W$, it is direct to check that \eqref{t1} is fulfilled provided
\begin{equation}\label{t4}
a^{m-1} \, \sigma m \left(1+ \left|\frac{\sigma m}{m-1} -2\right| + 2^{2-\sigma} \, C' \right) \leq \left(r^2+\rho^2\right)^{\frac{2-\sigma}2}  \quad \forall \rho\in (0, 1]
\end{equation}
and
\begin{equation}\label{t4a}
a^{m-1} \, \sigma m \left[ 1+  \left|\frac{\sigma m}{m-1} -2\right| + C' \left(1+\rho\right)^{2-\sigma} \right] \leq \left(r^2+\rho^2\right)^{\frac{2-\sigma}2} \quad \forall  \rho \ge 1 \, .
\end{equation}
If $ r $ ranges in the interval $ [1,\infty) $, then \eqref{t4}--\eqref{t4a} are fulfilled e.g.~by the choice
$$ a = \left[ \sigma m \left(1+ \left|\frac{\sigma m}{m-1} -2\right| + 2^{2-\sigma} \, C' \right) \right]^{-\frac{1}{m-1}} . $$
Inequality \eqref{t1} has therefore been established.
Besides, one sees that $W_{T,r} $, defined by \eqref{eq: expl-subsol}, satisfies
\[
W_{T,r} \geq (m-1)T \, \Delta\!\left( W_{T,r}^m \right) \quad \textrm{in } M \, ,
\]
which is equivalent to the fact that the separable profile $ \overline{u} $ satisfies \eqref{t1a}.
\end{proof}

As an immediate consequence of \eqref{e40-bis} and \eqref{eq: expl-subsol}, we have the following.
\begin{cor}\label{lem: supersol-datum}
Let $ u_0 \in X_{\infty,\sigma} $. Under the same assumptions and with the same notations as in Proposition \ref{supersol}, there holds
\begin{equation}\label{eq-sopra-dato}
\left|u_0(x)\right| \le W_{T,r}(x) \quad \textrm{for a.e. } x \in M
\end{equation}
with the choice
\begin{equation}\label{eq-sopra-T}
T= a^{m-1} \left\| u_0 \right\|_{\infty,r}^{1-m} \, .
\end{equation}
\end{cor}

\subsection{Proof the existence result}\label{sect: proofExi}
By means of the approximating problems \eqref{q10}, taking advantage of Proposition \ref{supersol} and the consequent Corollary \ref{lem: supersol-datum}, we are now able to prove Theorem \ref{thmexi}.
\begin{proof}[Proof of Theorem \ref{thmexi}]
For every $R>0$, let $u_R$ be the weak solution of problem \eqref{q10}, in the sense of Definition \ref{def2}. Clearly $u_R$ is also a very weak solution, namely
\begin{equation}\label{t10}
- \int_0^\infty \int_{B_R} u_R \, \phi_t \,  d\mu dt \,=\, \int_0^\infty \int_{B_R} u_R^m  \, \Delta \phi \, d\mu dt + \int_{B_R} u_0(x) \, \phi(x,0) \, d\mu
\end{equation}
for all $\phi\in C^\infty_c(B_R\times [0, \infty))$. So, let $ \overline{u} $ be the supersolution provided by Proposition \ref{supersol}, with $T$ chosen as in \eqref{eq-sopra-T}. Since $ \overline{u}>0 $, in view of \eqref{t1a} and \eqref{eq-sopra-dato} we can apply the standard comparison principle to get
\begin{equation}\label{t10-comp}
u_R \le \overline{u} = \left( 1 - \frac{t}{T} \right)^{-\frac{1}{m-1}} \left\| u_0 \right\|_{r,\infty} \left( r^2+ \rho(x)^2 \right)^{\frac{\sigma}{2(m-1)}} \quad \textrm{in } B_R \times (0,T) \, .
\end{equation}
On the other hand, because $ -\overline{u}^m = \left( - \overline{u} \right)^m $, there holds
\begin{equation*}\label{t1a-minus}
\left( -\overline{u} \right)_t \leq \Delta \left( -\overline{u} \right)^m \quad \textrm{in } M\times (0, T) \, .
\end{equation*}
Hence, an analogous application of the comparison principle on balls yields
\begin{equation}\label{t10-comp-minus}
- \overline{u} = - \left( 1 - \frac{t}{T} \right)^{-\frac{1}{m-1}} \left\| u_0 \right\|_{r,\infty} \left( r^2+ \rho(x)^2 \right)^{\frac{\sigma}{2(m-1)}} \le u_R  \quad \textrm{in } B_R \times (0,T) \, .
\end{equation}
As a consequence of \eqref{t10-comp} and \eqref{t10-comp-minus}, we have that
\begin{equation}\label{t10-abs}
\left| u_R(x,t) \right| \le \left( 1 - \frac{t}{T} \right)^{-\frac{1}{m-1}} \left\| u_0 \right\|_{r,\infty} \left( r^2+ \rho(x)^2 \right)^{\frac{\sigma}{2(m-1)}} \quad \textrm{for a.e. } (x,t) \in B_R \times (0,T) \, .
\end{equation}
Thanks to estimate \eqref{t10-abs}, a standard diagonal argument ensures that, up to subsequences, $ u_R $ (set to be zero in $ B_R^c \times (0,T) $) converges weakly$^\ast$ in $ L^\infty_{\rm loc}(M\times(0,T)) $ as $ R \to \infty $ to some function $ u \in L^\infty_{\rm loc}(M\times(0,T)) $. This is enough in order to pass to the limit in \eqref{t10} to infer that $u$ is a solution of \eqref{e64}, in the sense of Definition \ref{defsol}, satisfying \eqref{thm-limit-abs}. Finally, estimate \eqref{thm-norms} is a consequence of \eqref{thm-limit-abs} and \eqref{e40-norm}, whereas \eqref{thm-limsup} and \eqref{thm-tmax} follow from \eqref{thm-norms} and \eqref{n2}, respectively, upon letting $ r \to \infty $ and using \eqref{e40-limsup-final}. The last statement concerning initial data such that $ \lim_{\rho(x)\to\infty} \rho(x)^{-\frac{\sigma}{m-1}} \, | u_0(x) | = 0 $ is then a corollary of \eqref{thm-tmax}.
\end{proof}
%

\section{Uniqueness: proofs}\label{uniqueness}
Before giving the detailed proof of Theorem \ref{thmuniq}, let us describe from a general point of view its strategy. The latter is based on a so-called ``duality method'', which consists in choosing a suitable family of test functions in  the weak formulation of problem \eqref{e64}. In particular, for every $\varepsilon>0$ and $ n\in \mathbb N$, we shall take $ \phi \equiv \phi_\varepsilon \xi_n$, where $\{\phi_\varepsilon\}$ is a family of suitable cut-off functions and $\{\xi_n\}$ is a sequence of solutions of appropriate (dual) backward parabolic problems with unbounded coefficients $a_n$ which are associated with the difference of two possibly different solutions. Then, by letting $\varepsilon \to 0$, $n\to \infty$ and using a priori estimates on the test functions, we end up with
\begin{equation}\label{q11}
\int_M [u(T) - v(T)] \, \omega \, d\mu = 0
\end{equation}
for any $\omega \in C^\infty_c(M)$ with $ \omega\geq 0$ and all $T>0$ small enough. This easily implies that $u(t) \equiv v(t)$ for almost every $ t \in (0,T) $. In order to remove the constraint that $ T $ is suitably small, it is enough to perform a finite-step iteration.

We stress that the above strategy deeply relies on decay estimates for
\begin{equation}\label{eq: est-norm}
\sup_{S_R\times(0, T)} \left|\frac{\partial \xi_n}{\partial \nu}\right| ,
\end{equation}
where $\frac{\partial}{\partial \nu}$ denotes the outward normal derivative on $S_R$. To this aim, we shall construct a suitable supersolution $\eta$ (see \eqref{n21}).

\medskip

\begin{proof}[Proof of Theorem \ref{thmuniq}] Since both $u$ and $v$ satisfy \eqref{q50}, we get
\begin{equation}\label{n3}
\int_0^T \int_M \left[ (u - v) \, \xi_t + (u^m - v^m) \, \Delta \xi \right] d\mu dt\,=\, \int_M \left[ u(T) - v(T) \right]\xi(T) d\mu
\end{equation}
for any $\xi\in C^\infty_c(M\times [0, T])\,.$ Let
\begin{equation}\label{n4}
a(x,t):=
\begin{cases}
\frac{u^m(x,t) - v^m(x,t)}{u(x,t)- v(x,t)} & \textrm{if } u(x,t)\neq v(x,t) \, , \\
0 & \textrm{if } u(x,t)=v(x,t) \, .
\end{cases}
\end{equation}
It is easily seen that, thanks to \eqref{e20z}, there exists a constant $C_1>0$ such that
\begin{equation}\label{n5}
0\leq  a(x,t) \, \leq \, C_1 \, (1 + \rho(x))^{\sigma}\quad \textrm{for a.e. } (x,t)\in M\times (0, T)\,.
\end{equation}
Due to \eqref{n4}, identity \eqref{n3} reads
\begin{equation*}
\int_0^T \int_M (u - v) (\xi_t + a \, \Delta \xi)\, d\mu dt \,=\, \int_M [u(T)- v(T)] \, \xi(T) \, d\mu \, .
\end{equation*}
Now let $R_0>0$ and $R>R_0+1$. In the sequel, $R_0$ is meant to be fixed once for all, while $R$ will eventually go to infinity. Let
\begin{equation*}
\{a_{n, R}\}_{n\in \mathbb N} \equiv  \{a_n\}_{n\in \mathbb N}\subset C^\infty(M\times [0, T])\,,\quad a_n>0 \ \, \textrm{in}\ B_R \, , \quad \textrm{for any}\ n\in \mathbb N\,;
\end{equation*}
\begin{equation}\label{n8}
\omega\in C^\infty_c(M) \, , \quad 0 \leq \omega\leq 1 \, , \quad \omega(x)=0 \ \, \textrm{for all } \rho(x)\geq R_0\,.
\end{equation}
For every $n\in \mathbb N$, let $\xi_n$ be the solution of the backward parabolic problem
\begin{equation}\label{n9}
\begin{cases}
\partial_t \xi_n  + a_n \, \Delta \xi_n = 0 & \textrm{in } B_R \times (0, T) \, , \\
\xi_n = 0 & \textrm{on } \partial B_R\times (0, T) \, , \\
\xi_n = \omega & \textrm{in } B_R \times \{T\} \, .
\end{cases}
\end{equation}
Clearly there holds
\begin{equation}\label{n43}
\frac{\partial \xi_n}{\partial \nu}  \leq 0 \quad \textrm{on } \partial B_R\times (0, T)\,.
\end{equation}
For all $0<\varepsilon<\frac 1 2$ we can pick a family of nonincreasing cut-off functions $ \tilde \phi_\varepsilon $ such that
\begin{equation*}
\begin{aligned}
& \tilde \phi_\varepsilon\in C^\infty_c([0,\infty))\,,\quad 0\leq \tilde\phi_\varepsilon\leq 1\, , \\
& \tilde \phi_\varepsilon = 1\quad \textrm{in}\ [0, R-2\varepsilon)\,,\quad \tilde \phi_\varepsilon = 0 \quad \textrm{in}\ (R-\varepsilon, \infty)\,, \\
& |\tilde \phi'_\varepsilon|\, \leq\, \frac C{\varepsilon} \, \chi_{[R-2\varepsilon, R-\varepsilon]}\,,\quad |\tilde \phi'' |\,\leq\, \frac{C}{\varepsilon^2}\,\chi_{[R-2\varepsilon, R-\varepsilon]} \, ,
\end{aligned}
\end{equation*}
for some constant $C>0$ independent of $\varepsilon$. Set
\begin{equation*}
\phi_\varepsilon(x):=\tilde \phi_\varepsilon(\rho(x))\quad \textrm{for all}\ x\in M\,.
\end{equation*}
In particular, there holds
\begin{equation}\label{n12}
|\nabla \phi_\varepsilon| \leq \frac{C}{\varepsilon} \, \chi_{B_{R-\varepsilon}\setminus B_{R-2\varepsilon}}   \quad \textrm{in}\;\; M\,;
\end{equation}
furthermore, since $0\leq \sigma \leq 2$, from \eqref{e1}, \eqref{e5} and \eqref{eq: m-est-above} we deduce that
\begin{equation}\label{n13}
|\Delta \phi_\varepsilon|\,  \leq \frac{C}{\varepsilon^2} \, \chi_{B_{R-\varepsilon}\setminus B_{R-2\varepsilon}}\left(1+ R^{1-\sigma}\right)   \quad \textrm{in}\ M\,,
\end{equation}
for some other constant $C>0$ independent of $R$ and $\varepsilon$ (that shall not be relabelled below).

So, by plugging $\xi \equiv \phi_\varepsilon \xi_n$ in \eqref{n3}, we obtain:
\begin{equation}\label{n14}
\begin{aligned}
& \int_0^T \int_M (u - v) \phi_\varepsilon (a- a_n) \, \Delta \xi_n \, d\mu dt \,+ \, \int_0^T \int_M (u^m - v^m) (2\langle \nabla \phi_\varepsilon, \nabla \xi_n\rangle + \xi_n \, \Delta \phi_\varepsilon) \, d\mu dt \\
= &  \int_M [u(T)- v(T)] \, \omega \phi_\varepsilon \, d\mu\,.
\end{aligned}
\end{equation}
Let us set
\begin{equation}\label{n15}
I_{n \varepsilon}:= \int_0^T \int_M (u-v) \phi_\varepsilon (a- a_n) \, \Delta\xi_n \, d\mu dt\,,
\end{equation}
\begin{equation}\label{n16}
J_{n \varepsilon}:= \int_0^T \int_M (u^m - v^m) (2\langle \nabla \phi_\varepsilon, \nabla \xi_n\rangle + \xi_n \, \Delta \phi_\varepsilon) \, d\mu dt\,.
\end{equation}
In view of \eqref{n12}--\eqref{n13} we get
\begin{equation*}
|J_{n\varepsilon}| \leq C\left(1+ R^{1-\sigma}\right) \int_0^T \int_{B_R\setminus B_{R-2\varepsilon}} \left|u^m- v^m\right| \left(\frac{|\nabla \xi_n|}{\varepsilon}+\frac{|\xi_n|}{\varepsilon^2} \right) d\mu dt\,.
\end{equation*}
Since $\xi_n=0$ in $\partial B_R\times (0, T)$, we have that
\[ \sup_{ \substack{ R-2\varepsilon<\rho(x)<R \\  0<t<T} } |\xi_n(x,t)|\, \leq\, 2\varepsilon \sup_{ \substack{ R-2\varepsilon<\rho(x)<R \\ 0<t<T}}|\nabla \xi_n(x,t)|
\]
and
\[\lim_{\varepsilon \to 0} \, \sup_{ \substack{R-2\varepsilon<\rho(x)<R \\ 0<t<T}}|\nabla \xi_n(x,t)| = \sup_{ \substack{ \rho(x)=R \\ 0<t<T}} |\nabla \xi_n(x,t)| = \sup_{ \substack{ \rho(x)=R \\ 0<t<T }} \left|\frac{\partial \xi_n}{\partial \nu}(x,t)\right|.
\]
Hence,
\begin{equation}\label{n50}
\begin{aligned}
J_n := & \limsup_{\varepsilon\to 0} \left| J_{n \varepsilon} \right| \\
\leq & C \, T \left(1+ R^{1-\sigma}\right) \textrm{meas}(S_R)\, \sup_{S_R\times(0, T)} \left|\frac{\partial \xi_n}{\partial \nu}\right|  \limsup_{\varepsilon\to 0}
\sup_{ \substack{R-2\varepsilon<\rho(x)<R \\ 0<t<T}} \left|u^m(x,t)-v^m(x,t)\right| .
\end{aligned}
\end{equation}
By exploiting \eqref{e20z} both for $u$ and $v$, we infer
\begin{equation}\label{n50aa}
J_n \leq C \, T \left(1+ R^{1-\sigma}\right) \textrm{meas}(S_R)\, R^{\frac{\sigma m}{m-1}} \sup_{S_R\times(0, T)} \left|\frac{\partial \xi_n}{\partial \nu}\right| .
\end{equation}
On the other hand, for every $\varepsilon>0$ there holds
\begin{equation}\label{n19}
\left|I_{n \varepsilon}\right|\leq I_n:= \left(\int_0^T \int_{B_R} |u - v|^2 \, \frac{(a - a_n)^2}{a_n} \, d\mu dt \right)^{\frac 1 2} \left(\int_0^T \int_{B_R} a_n \, (\Delta \xi_n)^2 \, d\mu dt \right)^{\frac 1 2}\,.
\end{equation}

Now we need to estimate \eqref{eq: est-norm}. To this end, suppose that for all $R>R_0+1$ there exists $n_0=n_0(R)$ such that for every $n\in \mathbb N , n>n_0$
\begin{equation}\label{n20}
a_n \leq C_2 \, (1+ \rho(x))^{{\sigma}}\quad \textrm{for all}\ (x,t)\in B_R\times (0, T)\,,
\end{equation}
for some $C_2>0$ independent of $n$ and $R$. We shall comment later on such assumption.

\medskip

\noindent (a) Consider first the case where $-2<\gamma\leq 2$, so that $0\leq \sigma<2$. Take any $\lambda>0$, $K>0$, $0<t_0<\frac T4$, $\rho_0>1$ and set
\begin{equation}\label{n21}
\eta(x,t)\equiv \eta(\rho(x), t):= \lambda \, e^{\frac{K}{t-T-t_0}(\rho(x) + \rho_0)^{2-\sigma}}\quad \textrm{for all}\ (x,t)\in B_{R_0}^c\times [0, T]\,.
\end{equation}

We claim that, for a suitable choice of the parameters $K$ and $\rho_0$, the function $\eta$ satisfies
\begin{equation}\label{n33}
\eta_t \,+\, a_n \, \Delta \eta \leq 0 \quad \textrm{in}\  (B_{R} \setminus \bar B_{R_0}) \times (0, T)\,.
\end{equation}
In fact, for every $\rho>1$ and $ 0<t<T$ we have
\begin{equation}\label{n22}
\eta_t(\rho, t)=-\frac{K (\rho+\rho_0)^{2-\sigma}}{(t-T-t_0)^2} \, \eta(\rho, t) < 0\,,
\end{equation}
\begin{equation*}
\eta_\rho(\rho, t)=\frac{  K (2-\sigma) (\rho+\rho_0)^{1-\sigma}}{t-T-t_0}\,\eta(\rho, t) < 0
\end{equation*}
and
\begin{equation*}
\eta_{\rho\rho}(\rho, t)=\left[\frac{ K^2 (2-\sigma)^2 (\rho+\rho_0)^{2(1-\sigma)}}{(t-T-t_0)^2} + \frac{ K (2-\sigma)(1-\sigma)(\rho+\rho_0)^{-\sigma}}{t-T-t_0}\right]\eta(\rho, t)\,.
\end{equation*}
Let $\mathcal D:= \{(x,t)\in (B_R \setminus \bar B_{R_0})\times (0, T) : \, \Delta \eta(x,t)<0\}$. Thanks to \eqref{n22} and \eqref{n5} it follows that
\begin{equation*}
\eta_t + a_n \, \Delta \eta \leq 0 \quad \textrm{in}\ \mathcal D\,.
\end{equation*}
In the region $\{ (x,t)\in [ (B_R \setminus \bar B_{R_0}) \times (0, T)] \} \setminus \mathcal D$, in view of \eqref{e1}, \eqref{e5} and \eqref{n20}, there holds
\begin{equation}\label{n26}
\begin{aligned}
& \eta_t(\rho, t) + a_n\,\Delta \eta(\rho, t) \\
\leq & \eta_t(\rho, t) + C_2 (\rho_0+\rho)^{{\sigma}}\Delta \eta(\rho,t) \\
\leq & \frac{\eta(\rho, t)(\rho+\rho_0)^{2-\sigma}}{(t-T-t_0)^2} \, [ -K +K^2 C_2 (2-\sigma)^2 + K C_2(2-\sigma)(1-\sigma)(t-T-t_0)(\rho+\rho_0)^{\sigma-2} \\
&  + K C_2 (2-\sigma)(t-T-t_0)(\rho+\rho_0)^{-1+\sigma} \, m(\rho, \theta) ] \\
\leq & \frac{\eta(\rho, t)(\rho+\rho_0)^{2-\sigma}}{(t-T-t_0)^2} \, [-K + K^2 C_2 (2-\sigma)^2 + K C_2 (2-\sigma)(1-\sigma)(t-T-t_0)(\rho+\rho_0)^{\sigma-2} ] \, .
\end{aligned}
\end{equation}
If $\sigma \in [0,1]$, then from \eqref{n26} we deduce that
\begin{equation}\label{n30}
\eta_t + a_n \, \Delta \eta \leq 0 \quad \textrm{in } \big[ (B_R \setminus \bar B_{R_0}) \times (0, T) \big] \setminus \mathcal D
\end{equation}
provided
\begin{equation*}
0< K < \frac{1}{C_2(2-\sigma)^2}\,.
\end{equation*}
On the other hand, if $1<\sigma<2,$ then \eqref{n30} is under the condition
\begin{equation*}
K < \frac{1}{2 C_2(2-\sigma)^2}\,,\quad \rho_0>\big[ 2 (2-\sigma)(\sigma-1)(T+t_0) C_2\big] ^{\frac 1{2-\sigma}}\,.
\end{equation*}
Thus the claim has been shown.

By selecting
\[
\lambda \geq \| \omega\|_\infty \, e^{\frac K{t_0}(R_0+\rho_0)^{2-\sigma}}
\]
we obtain that
\begin{equation}\label{n33-bis}
\eta \geq \| \omega\|_\infty \, \chi_{B_{R_0}} \geq \omega \quad \textrm{in } [\bar B_{R}\times \{T\}] \cup[\partial B_{R_0}\times (0, T)]\,.
\end{equation}
Hence, in view of \eqref{n33} and \eqref{n33-bis}, $\eta$ is a supersolution of problem (note that $ \xi_n \le \| \omega \|_\infty $)
\begin{equation}\label{n9b}
\begin{cases}
v_t  + a_n \, \Delta v = 0 & \textrm{in }\ (B_R \setminus \bar B_{R_0}) \times (0, T) \, , \\
v = 0 & \textrm{on } \partial B_R\times (0, T) \, , \\
v = \xi_n & \textrm{on } \partial B_{R_0} \times (0, T) \, ,\\
v = 0 & \textrm{in } [ \bar B_R\setminus B_{R_0}] \times \{T\} \,.
\end{cases}
\end{equation}
By the comparison principle, we infer that for every $n\in \mathbb N$
\begin{equation}\label{n38}
\eta \geq \xi_n\quad \textrm{in } [B_R\setminus B_{R_0}]\times (0, T)\,.
\end{equation}
Let us choose $k_1>0 $ and $ k_2<0$ such that
\begin{equation}\label{n34}
\begin{cases}
\frac{k_1}{(R-1)^{N-2}} + k_2 = \eta (R-1, 0) \, ,  & 
\\
\frac{k_1}{R^{N-2}}+k_2=0 \, ,  & 
\end{cases}
\end{equation}
namely
\[
k_1 = \frac{R^{N-2}(R-1)^{N-2}}{R^{N-2}- (R-1)^{N-2}} \, \eta(R-1, 0)\,, \quad k_2 = - \frac{(R-1)^{N-2}}{R^{N-2}- (R-1)^{N-2}} \, \eta(R-1, 0) \, .
\]

Assume for the moment that $N\geq 3$. Let
\begin{equation}\label{e1lu}
h(x)\equiv h(\rho(x)):= \frac{k_1}{\rho(x)^{N-2}} + k_2\quad \textrm{for all } x \in \bar B_R\setminus B_{R-1}\,.
\end{equation}
In view of \eqref{n34}, we have that
\begin{equation}\label{n35}
h = 0 = \xi_n \quad \textrm{in}\ \partial B_R\times (0, T]\,.
\end{equation}
By \eqref{n22} and \eqref{n38}--\eqref{n34}, there holds
\begin{equation*}\label{n36}
h(x)=\eta(R-1, 0)\geq \eta(R-1, t) \geq \xi_n\quad \textrm{for all } (x,t)\in \partial B_{R-1}\times (0, T]\,.
\end{equation*}
Since $R>R_0+1$, in view of \eqref{n8} it follows that
\begin{equation*}\label{n37}
h(x) \geq 0 = \omega(x)=\xi_n(x) \quad \textrm{for all}\  x \in \bar B_R\setminus B_{R-1}\,.
\end{equation*}
Furthermore, since $h'<0$, from \eqref{e1} and \eqref{e5} we get
\begin{equation}\label{n40}
\Delta h(x) \leq h''(\rho) + \frac{N-1}{\rho} \, h'(\rho) = 0 \quad \textrm{for all } x\in M\setminus \{o\}\,.
\end{equation}
From \eqref{n35}--\eqref{n40} we can then deduce that for each $n\in \mathbb N$ the function $h$ is a supersolution of problem
\begin{equation}\label{n41}
\begin{cases}
v_t  + a_n \, \Delta v = 0 & \textrm{in } (B_R\setminus \bar B_{R-1}) \times (0, T) \, , \\
v = 0 & \textrm{on } \partial B_R\times (0, T) \, , \\
v = \xi_n & \textrm{on } \partial B_{R-1} \times (0, T) \, , \\
v = 0  & \textrm{in } [\bar B_R\setminus B_{R-1}]\times \{T\}\,.
\end{cases}
\end{equation}
On the other hand $ \xi_n$ is a solution of \eqref{n41}. Hence, by the comparison principle,
\begin{equation}\label{n42}
h \geq \xi_n \quad \textrm{in } [B_R\setminus B_{R-1}]\times (0, T)\,.
\end{equation}
As a consequence, \eqref{n35} and \eqref{n42} imply
\begin{equation}\label{n44}
\frac{\partial}{\partial \nu}(h - \xi_n)\leq 0\quad \textrm{on } \partial B_R\times (0, T)\,.
\end{equation}
Due to \eqref{n43} and \eqref{n44} we therefore obtain
\begin{equation}\label{n45}
\left|\frac{\partial \xi_n}{\partial \nu} \right| \leq \left|\frac{\partial h}{\partial \nu} \right| \quad \textrm{on } \partial B_R\times (0, T) \, ;
\end{equation}
note that
\begin{equation}\label{e2lu}
\frac{\partial h(x)}{\partial \nu} =  \frac{(2-N)k_1}{\rho(x)^{N-1}}= \frac{2-N}{R^{N-1}} \, \frac{R^{N-2}(R-1)^{N-2}}{R^{N-2}- (R-1)^{N-2}} \, \eta(R-1, 0)\quad \textrm{for all } x\in \partial B_R\,.
\end{equation}
Thus there exists a constant $\hat C=\hat C(N, R_0)>0$ such that for any $R>R_0+1$
\begin{equation}\label{n46}
\left| \frac{\partial h(x)}{\partial \nu} \right| \leq \hat C \, \eta(R-1, 0)= \hat C \, \lambda \, e^{-\frac{K}{T+t_0}(R -1+ \rho_0)^{2-\sigma}}\quad \textrm{for all } x\in \partial B_R\,.
\end{equation}
Furthermore, since $0\leq \sigma<2$, hypothesis \eqref{H} implies that \eqref{e3c} holds for some $\psi\in \mathcal A$ such that $\log[\psi(\rho)]\sim \rho^{2-\sigma}$ as $\rho\to \infty$ (recall the discussion in Section \ref{main}). Hence, \eqref{n50a} yields
\begin{equation}\label{e52z}
\operatorname{meas}(S_R) \leq e^{\tilde C R^{2-\sigma}} \quad \textrm{for any } R>1
\end{equation}
for some $\tilde C>0$. From \eqref{n50aa}, \eqref{n45}, \eqref{n46} and \eqref{e52z} we then get
\begin{equation*}
J_n \leq C \, \hat C \, T \left(1+ R^{1-\sigma}\right) e^{\tilde C R^{2-\sigma}} \, R^{\frac{\sigma m}{m-1}} \, \lambda \, e^{-\frac{K}{T+t_0}(R-1+ \rho_0)^{2-\sigma}} \, ,
\end{equation*}
\begin{equation}\label{n52}
\limsup_{R\to \infty} \, \limsup_{n\to \infty} J_n =  0\,,
\end{equation}
provided
\begin{equation}\label{eq: T-small}
0<T<\frac{4 K}{5 \tilde C} \, .
\end{equation}

In order to estimate $I_n$ we multiply the differential equation in \eqref{n9} by $\Delta\xi_n$ and integrate over $B_R\times (0, T)$ to get
\[
\frac 1 2\int_M | \nabla\xi_n(0) |^2 \, d\mu + \int_0^T \int_M a_n \, (\Delta \xi_n)^2 \, d\mu dt \,=\, \frac 1 2 \int_M |\nabla \omega|^2 \, d\mu \, ,
\]
so that
\begin{equation}\label{n48}
I_n \leq C(\omega,u,v,R)\left[\int_0^T \int_{B_R}\frac{(a-a_n)^2}{a_n} \, d\mu dt  \right]^{\frac 1 2}\,.
\end{equation}
Let $\bar a$ be the extension of $a$ to $M\times \mathbb R$, supposed to be zero in $M\times [(-\infty,0)\cup (T, \infty)]$. Using standard mollifiers, for every $R>R_0+1$ we can construct a sequence $\{\alpha_{n, R} \}\equiv \{\alpha_n\}\subset C^\infty(M\times \mathbb R)$, $ \alpha_n>0$ such that
\begin{equation}\label{n48-bis}
\int_0^T \int_{B_R} (\bar a - \alpha_n)^2 \, d\mu dt \leq \frac 1 {n^2}
\end{equation}
and \eqref{n20} (with $a_n$ replaced by $\alpha_n$) hold with $ C_2 =C_1+1 $ for all $ n>n_0 =n_0(R)\in \mathbb N$. 
Upon setting
\begin{equation}\label{n48-ter}
a_n:= \alpha_n + \frac 1n\,,
\end{equation}
we can then assume that $a_n$ satisfies \eqref{n20} with $C_2 \equiv C_2+1$.
%
%
Note that, from \eqref{n48-bis}--\eqref{n48-ter}, there holds
\begin{equation}\label{n49}
\left(\int_0^T \int_{B_R} \frac{(a-a_n)^2}{a_n} \, d\mu dt \right)^{\frac 1 2} \leq \sqrt{2n}\left( \frac 1 n + \frac{\sqrt{T \mu(B_R)}}{n} \right) .
\end{equation}
Therefore, for any fixed $R>R_0+1$, in view of \eqref{n48} and \eqref{n49} we obtain
\begin{equation}\label{n53}
\limsup_{n\to \infty} I_n = 0\,.
\end{equation}

\smallskip
Now, for each $R>R_0+1$ we let first $\varepsilon\to 0$ and then $n\to \infty$ in \eqref{n14}. By exploiting \eqref{n16}, \eqref{n50}, \eqref{n52}, \eqref{n15}, \eqref{n19} and \eqref{n53} we infer the validity of \eqref{q11}, upon letting $R\to \infty$ eventually. Since $\omega$ is arbitrary, we deduce that $u(T)=v(T)$. We point out that equality holds for all $T$ complying with \eqref{eq: T-small}: however, such restriction can be removed by repeating iteratively the above scheme of proof, since the constant on the r.h.s.~of \eqref{eq: T-small} only depends on initial data through the constant $C_1$ in \eqref{n5}.

\smallskip

Let us now briefly discuss the case that $N=2$. We shall replace the function $h$ defined in \eqref{e1lu} by the following one:
\[
h(x) := - \bar k_1 \, \log(\rho(x)) + \bar k_2 \quad \textrm{for all } x \in \bar B_R\setminus B_{R-1} \,,
\]
where
\[\bar k_1 = \frac{\eta(R-1, T)}{\log(R) - \log(R-1)} > 0 \,, \quad \bar k_2=\frac{\log R}{\log(R)-\log(R-1)} \, \eta(R-1, T) >0\,,\]
so that
\[
\begin{cases}
-\bar k_1 \log(R-1)+ \bar k_2 = \eta (R-1, 0) \, , \\
-\bar k_1 \log(R)+\bar k_2=0  \, .
\end{cases}
\]
Hence, \eqref{n35}--\eqref{n40} continue to hold. Moreover, in place of \eqref{e2lu} we have
\[\frac{\partial h(x)}{\partial \nu} = - \frac{\bar k_1}{R} = -\frac{\eta(R-1, T)}{R\,[\log(R)-\log(R-1)]} \quad \textrm{for all } x\in \partial B_R \, , \]
which implies that also \eqref{n46} is fulfilled. The conclusion then follows as in the case $N\geq 3$.

\medskip

\noindent (b) Suppose now that $\gamma=-2$, whence $\sigma =2$. For any $\alpha,\beta>0$, set
\begin{equation*}
\eta(x,t)\equiv \eta(\rho(x), t) := \lambda \, e^{\alpha(T-t)}\left(1 + \rho(x)^2\right)^{-\beta}\quad \textrm{for all } (x,t)\in M\times [0, T]\,.
\end{equation*}
We claim that, for a suitable choice of the parameters $\alpha$ and $\beta$, the function $\eta$ satisfies
\begin{equation}\label{n33b}
\eta_t \,+\, a_n \, \Delta \eta \leq 0 \quad \textrm{in } B_R \times (0, T)\,.
\end{equation}
In fact, for every $\rho>0$ and $ 0<t<T$, we have:
\begin{equation}\label{n22b}
\eta_t(\rho, t)=-\alpha \, \eta(\rho, t)< 0\,,
\end{equation}
\begin{equation*}
\eta_\rho(\rho, t) = - 2 \beta \, \frac{\rho}{1+\rho^2} \, \eta(\rho, t) < 0\,,
\end{equation*}
\begin{equation*}
\eta_{\rho\rho}(\rho, t)= 2 \beta \, \frac{\eta(\rho, t)}{(1+\rho^2)^2}\left[- 1 + (1+2\beta)\rho^2 \right] .
\end{equation*}
Let $\mathcal D:= \{(x,t)\in B_R \times (0, T) : \, \Delta \eta(x,t)<0\}$. Thanks to \eqref{n5} and \eqref{n22b} it therefore follows that
\begin{equation*}
\eta_t + a_n \, \Delta \eta \leq 0 \quad \textrm{in } \mathcal D\,.
\end{equation*}
Note that, in view of \eqref{n20}, for some $C_3>0$ depending only on $C_2$ there holds
\begin{equation}\label{n20b}
a_n\leq C_3 \left(1+\rho(x)^2\right)\quad \textrm{in } B_R \times (0, T)
\end{equation}
for every $n\in \mathbb N, n>n_0$. In the region $ \{ (x,t) \in  B_R \times (0, T) \} \setminus \mathcal D$, in view of \eqref{e1}, \eqref{e5} and \eqref{n20b}, we have
\begin{equation*}
\begin{aligned}
\eta_t(\rho, t) + a_n \, \Delta \eta(\rho, t)  \le & \eta_t(\rho, t) + C_3 \left(1 +\rho^2\right) \Delta \eta(\rho,t) \\
\leq & \eta(\rho, t) \left\{-\alpha + 2 \beta C_3\left[-\frac 1{1+\rho^2}+ (1+2\beta) \, \frac{\rho^2}{1+\rho^2}-  m(\rho, \theta) \, \rho\right] \right\} \\
\leq & \eta(\rho, t) \left[-\alpha + 2 \beta C_3 \, (1+2\beta)   \right] \\
\leq & 0 \, ,
\end{aligned}
\end{equation*}
provided $ \alpha > 2\beta C_3 (1 + 2\beta )$. Thus \eqref{n33b} holds. The choice
\[ \lambda \geq \| \omega\|_\infty  \left(1+ R_0^2\right)^\beta \]
makes $\eta$ a supersolution of problem \eqref{n9b}. The conclusion then follows by arguing as in case (a); the only difference lies in the proof of \eqref{n52}. In fact, in view of \eqref{H} with $\gamma=-2$, \eqref{e3c} is satisfied for some $\psi \in \mathcal A$ satisfying \eqref{e51z}. As a consequence,
\begin{equation}\label{q71}
\textrm{meas}(S_R)\, \leq\, \tilde C R^{(N-1)\delta}\quad \textrm{for any } R>1
\end{equation}
for some $\tilde C>0$ and $\delta$ as in \eqref{e51z}. Whence, from \eqref{n50aa} and \eqref{q71} we get that
\begin{equation}\label{n18b}
J_n \leq 2C \, T \, \tilde C \, R^{(N-1)\delta} R^{\frac{2 m}{m-1}} \sup_{S_R\times(0, T)} \left|\frac{\partial \xi_n}{\partial \nu}\right| \,.
\end{equation}
Observe that, similarly to \eqref{n46}, we have
\begin{equation}\label{n46b}
\left| \frac{\partial h(x)}{\partial \nu} \right|\leq \hat C \, \eta(R-1, 0)= \hat C \, \lambda \, e^{\alpha T} \left(1+ (R-1)^2\right)^{-\beta}\quad \textrm{for all } x \in \partial B_R\,.
\end{equation}
By virtue of \eqref{n45}, \eqref{n18b} and \eqref{n46b}, we end up with
\begin{equation*}
J_n \leq 2C \, \hat C \, T \, \tilde C \, R^{(N-1)\delta+\frac{2 m}{m-1}} \, \lambda \, e^{\alpha T} (1+ (R-1)^2)^{-\beta}\,.
\end{equation*}
Hence, it is plain that \eqref{n52} is satisfied provided
\[\beta > \frac{m}{m-1} + \frac{(N-1)\delta}{2}\,.\]

\medskip

\noindent(c) In the cases $\gamma<-2$ (where again $ \sigma=2 $), one follows verbatim the argument given in item (b), with $\delta$ replaced by 1 (recall the discussion in Section \ref{main}).
\end{proof}

By means of minor modifications in the proof of Theorem \ref{thmuniq}, we can also obtain the following comparison principle.
\begin{cor}\label{comppr}
Let assumption \eqref{H} be satisfied. Let $u$ and $v$ be a subsolution and a supersolution, respectively, of problem \eqref{e64} (with the same $ u_0 $ and $T$). Suppose that, for some $C>0$, \eqref{e20z} holds both for $u$ and for $v$. Then $u\leq v$ a.e.~in $M\times (0, T)$\,.
\end{cor}

\section{Maximal existence time, nonexistence and blow-up: proofs}\label{blup}

We start off this section by the analogue of Lemma \ref{prop-comp-ricci}, under suitable assumptions over \emph{sectional} curvatures.

\begin{lem}\label{prop-comp-sect}
Let assumption \eqref{H}-\textnormal{(i)} be satisfied. Assume in addition that
\begin{equation*}\label{assumption-sectional-prop}
\operatorname{K}_{\omega}(x)\leq - C_1 \, d(x,o)^\gamma \quad \forall x \in M \setminus B_{R_1}
\end{equation*}
for some $ \gamma \in (-2,2] $ and $ C_1 , R_1 > 0 $. Then there exists a positive constant $ C^{\prime\prime} $, depending on $ C_1,\gamma,N $, such that
\begin{equation}\label{eq: m-est-below}
m(\rho,\theta) \ge \frac{C^{\prime\prime}}{\rho\left( 1+\rho \right)^{\sigma-2}} \quad \forall x \equiv (\rho,\theta) \in M \setminus \{ o \} \, .
\end{equation}
\end{lem}
\begin{proof}
There are not major differences with respect to the proof of Lemma \ref{prop-comp-ricci}. Again, we give some details for the reader's convenience. One has to provide a function $ \psi \in \mathcal{A} $ complying with
\begin{equation}\label{eq: m-est-below-diffeq-a}
\frac{\psi''(\rho)}{\psi(\rho)}  \leq  C_1 \, \rho^\gamma  \quad \forall \rho > R_1 \, ,
\end{equation}
\begin{equation}\label{eq: m-est-below-diffeq-b}
\frac{\psi''(\rho)}{\psi(\rho)}  \leq  0  \quad \forall \rho \in (0,R_1]
\end{equation}
and satisfying
\begin{equation}\label{eq: m-est-below-diffeq-2}
(N-1) \, \frac{\psi'(\rho)}{\psi(\rho)} \ge \frac{C^{\prime\prime}}{\rho\left( 1+\rho \right)^{\sigma-2}} \quad \forall \rho>0 \, .
\end{equation}
The same issues as in the case of \eqref{eq: m-est-above-diffeq} occur here. Actually the differential inequality is a bit more rigid because of the constraint \eqref{eq: m-est-below-diffeq-b}, which is fundamental in order to match sectional curvature in $B_{R_1}$. Nevertheless, it is still possible to show existence of some $ \psi \in \mathcal{A} $ fulfilling \eqref{eq: m-est-below-diffeq-a}--\eqref{eq: m-est-below-diffeq-2}. We refer again to \cite{GMV}: see in particular Lemma 4.1 there.
\end{proof}

Before proving Theorem \ref{opt1} and the consequent Corollary \ref{opt2}, we need two auxiliary lemmas concerning subsolutions to an elliptic problem deeply related to \eqref{e64}.
\begin{lem}\label{lem: exist-max}
Let assumption \eqref{H} be satisfied. For all $T>0$, suppose there exists a nonnegative, nontrivial function $ V_T \in X_{\infty,\sigma} $ satisfying
\begin{equation}\label{subsol-ellliptic}
V_{T} \le (m-1) T \, \Delta\left(V_T^m\right) \quad \textrm{in } M \, .
\end{equation}
Let $ u $ be any nonnegative solution to problem \eqref{e64}, in the sense of Definition \ref{defsol}, with initial datum $ u_0 \ge V_T $. Then the maximal existence time for $ u $ is at most $T$.
\end{lem}
\begin{proof}
Let us define
$$ \underline{u}(x,t):=\left( 1 - \frac{t}{T} \right)^{-\frac{1}{m-1}} V_T(x) \quad \forall (x,t) \in M \times [0,T) \, . $$
In view of \eqref{subsol-ellliptic}, it is straightforward to check that $ \underline{u} $ is a subsolution to \eqref{e64}. Moreover, $ \underline{u}(t)$ belongs to $ X_{\infty,\sigma} $ for all $ t \in [0,T) $ and
\begin{equation}\label{subsol-ellliptic-blow}
\lim_{t \to T^{-}} \underline{u}(x,t) = + \infty \quad \forall x \in \mathcal{P} \, ,
\end{equation}
where $ \mathcal{P} $ is the positivity set of $ V_T $ which, by assumption, is of nonzero measure.

Suppose by contradiction that the maximal existence time for $u$ is larger than $ T $. Let us denote by $ \tau $ the maximal time for which $u$ is bounded from below by $ \underline{u} $, namely the largest number such that
\begin{equation}\label{subsol-ellliptic-under}
u(x,t) \ge \underline{u}(x,t) \quad \mathrm{in}\ M \times (0,\tau) \, .
\end{equation}
We set $ \tau=0 $ in case such a time does not exist. It is apparent that, thanks to \eqref{subsol-ellliptic-blow}, a contradiction is achieved if we show that $ \tau $ cannot be smaller than $T$. Hence, in order to show that $\tau\ge T$, suppose by contradiction that $ \tau < T $. In view of \eqref{subsol-ellliptic-under} there holds
\begin{equation}\label{subsol-ellliptic-under-1}
u(x,\tau) \ge \underline{u}(x,\tau) \quad \mathrm{in}\ M \, .
\end{equation}
Let us define by $ \hat{u} $ the solution of \eqref{e64} with initial datum $ \underline{u}(\tau) $. Since $ \underline{u}(\tau) \in X_{\infty,\sigma} $, Theorem \ref{thmexi} ensures that $ \hat{u} $ exists in $ X_{\infty,\sigma} $ for some time $ \varepsilon>0 $. We can assume with no loss of generality that $ \varepsilon < T-\tau $. By means of Corollary \ref{comppr} we can therefore deduce that
\begin{equation}\label{subsol-ellliptic-under-2}
\hat{u}(x,t-\tau) \ge \underline{u}(x,t) \quad \mathrm{in}\ M \times (\tau,\tau+\varepsilon) \, .
\end{equation}
Now, for every $R>0$, let us introduce the solution $ \hat{u}_{R}$ of the homogeneous Dirichlet problem
\begin{equation}\label{a-mm26}
\begin{cases}
(\hat{u}_{R})_t  \,=\, \Delta (\hat{u}_{R}^m) & \textrm{in } B_R\times (0,\infty)\,, \\
\hat{u}_{R} \,=\, 0 & \textrm{on } \partial B_R \times (0, \infty )\,, \\
\hat{u}_{R} \,=\, \underline{u}(\tau) \rfloor_{B_R} & \textrm{in } B_R\times \{0\}\,.
\end{cases}
\end{equation}
By comparison principles for very weak solutions on balls (see e.g.~\cite{ACP,Vaz07}) we have that
\begin{equation}\label{a-mm28}
0\leq \hat{u}_{R_1} \leq \hat{u}_{R_2} \leq \hat{u} \quad \textrm{in } B_{R_1} \times (0 , \varepsilon )
\end{equation}
for all $ 0<R_1<R_2 $. The same comparison principles (recall \eqref{subsol-ellliptic-under-1} and that $ u $ exists beyond $T$) ensure that
\begin{equation}\label{a-mm29}
\hat{u}_{R}(x,t-\tau) \leq u(x,t) \quad \textrm{in } B_{R} \times (\tau , \tau+\varepsilon )
\end{equation}
for all $ R>0 $. Thanks to \eqref{a-mm28} and to the uniqueness Theorem \ref{thmuniq} it is direct to check that in fact
\begin{equation*}\label{a-mm30}
\lim_{R\to\infty} \hat{u}_R(x,t) = \hat{u}(x,t) \quad \textrm{in } M \times (0,\varepsilon) \, ,
\end{equation*}
which, combined with \eqref{a-mm29} and \eqref{subsol-ellliptic-under-2}, yields
\begin{equation}\label{a-mm29-bis}
\underline{u}(x,t) \leq u(x,t) \quad \textrm{in } M \times (\tau , \tau+\varepsilon ) \, .
\end{equation}
A contradiction is then achieved since \eqref{a-mm29-bis} is incompatible with the definition of $ \tau $.
\end{proof}

\begin{lem}\label{lem: exist-subsol}
Let assumption \eqref{H}-\textnormal{(i)} be satisfied and let $ \gamma \in (-\infty,2) $. If $ \gamma \in (-2,2) $, assume in addition that there exist $ C_1 , R_1 > 0 $ such that
\begin{equation}\label{assumption-sectional}
\operatorname{K}_{\omega}(x)\leq - C_1 \, d(x,o)^\gamma \quad \forall x \in M \setminus B_{R_1}  \, .
\end{equation}
Then for all $T>0$ there exists a regular, positive function $ V_T \in X_{\infty,\sigma} $ which satisfies \eqref{subsol-ellliptic}. More precisely, one can choose $V_T \equiv W_{T,r}$ with $ W_{T,r} $ as in \eqref{eq: expl-subsol}, for suitable positive constants $ a,r $ depending only on $C_1,R_1,\gamma,N,m$.
\end{lem}
\begin{proof}
Let us consider first the case $-2 <\gamma <2$, so that $0< \sigma <2$. In view of assumptions \eqref{H}-\textnormal{(i)} and \eqref{assumption-sectional}, inequality \eqref{e5} and Lemma \ref{prop-comp-sect} imply
\begin{equation}\label{mm4}
m(\rho, \theta) \geq \frac{N-1}{\rho} \quad \textrm{for all }  \rho \in (0,1] \, , \ \theta \in \mathbb S^{N-1} \, ,
\end{equation}
and
\begin{equation}\label{mm5}
m(\rho,\theta) \geq \frac{C^{\prime\prime}}{\rho^{\sigma-1}} \quad \textrm{for all } \rho>1 \, , \ \theta\in \mathbb S^{N-1} \, ,
\end{equation}
where the positive constant $ C^{\prime\prime} $ is as in \eqref{eq: m-est-below}.

We pick our candidate subsolution to be radial. More precisely, let $W(x)\equiv W(\rho(x))$ be defined by \eqref{mm18}.
Note that we can choose $r>0$, only depending on $C^{\prime\prime},\sigma,N,m $, so that
\begin{equation}\label{mm8}
\frac{N-1}{2} \, \frac{\sigma m}{m-1} \left( r^2 + \rho^2 \right)^{\frac{\sigma m}{2(m-1)}-1} \geq \frac{\sigma m}{m-1}\left|\frac{\sigma m}{m-1} -2 \right| \rho^2 \left( r^2 + \rho^2 \right)^{\frac{\sigma m}{2(m-1)}-2} \quad \forall\, \rho\in (0,1]
\end{equation}
and
\begin{equation}\label{mm9}
\frac{C^{\prime\prime}}{2} \, \frac{\sigma m}{m-1} \, \rho^{2-\sigma} \left(r^2 + \rho^2\right)^{\frac{\sigma m}{2(m-1)}-1} \geq \frac{\sigma m}{m-1}\left|\frac{\sigma m}{m-1} -2 \right| \rho^2 \left( r^2 + \rho^2 \right)^{\frac{\sigma m}{2(m-1)}-2} \quad \forall \rho \ge 1 \, .
\end{equation}
Indeed, one can take e.g.
\begin{equation*}
r = \sqrt{\frac 2 {(N-1) \wedge C^{\prime\prime} } \left| \frac{\sigma m}{m-1} -2 \right| } \vee \left(\frac{2-\sigma}{C^{\prime\prime}} \right)^{\frac{1}{2-\sigma}} \left(\frac{\sigma}{2-\sigma} \right)^{\frac{\sigma}{2(2-\sigma)}}\left|\frac{\sigma m}{m-1} - 2\right|^{\frac 1{2-\sigma}}\,.
\end{equation*}
From \eqref{e1}, \eqref{mm7}, \eqref{mm4}--\eqref{mm5}, \eqref{mm8}--\eqref{mm9} and the fact that $ W(\rho) $ is nondecreasing, we can therefore deduce that
\begin{equation}\label{mm11}
\Delta(W^m)(x) \geq  a^m \, \frac{N-1}{2} \, \frac{\sigma m}{m-1} \left(r^2 + \rho^2\right)^{\frac{\sigma m}{2(m-1)}-1}  \quad \textrm{for all } x\equiv (\rho, \theta) \in B_1
\end{equation}
and
\begin{equation}\label{mm12}
\Delta(W^m)(x) \geq a^m\, \frac{C^{\prime\prime}} 2 \, \frac{\sigma m}{m-1} \, \rho^{2-\sigma} \left(r^2 + \rho^2\right)^{\frac{\sigma m}{2(m-1)}-1}  \quad \textrm{for all } x \equiv (\rho, \theta) \in B_1^c \, .
\end{equation}
Now we want to select the parameter $a>0$ in order to make $W$ satisfy
\begin{equation}\label{mm13}
W \leq (m-1) \, \Delta (W^m) \quad \textrm{in } M \, .
\end{equation}
To this purpose, due to \eqref{mm11}--\eqref{mm12} and to the regularity of $W$, it suffices to require that
\begin{equation}\label{mm14}
a \left( r^2+ \rho^2 \right)^{\frac{\sigma}{2(m-1)}} \leq a^m \, \frac{N-1}{2} \, \sigma m \left(r^2 + \rho^2\right)^{\frac{\sigma m}{2(m-1)}-1} \quad \forall \rho\in (0,1]
\end{equation}
and
\begin{equation}\label{mm15}
a \left(r^2+ \rho^2\right)^{\frac{\sigma}{2(m-1)}} \leq a^m \, \frac{C^{\prime\prime}}{2} \, \sigma m \,  \rho^{2-\sigma} \left(r^2 + \rho^2\right)^{\frac{\sigma m}{2(m-1)}-1} \quad \forall \rho \geq 1 \, .
\end{equation}
Conditions \eqref{mm14} and \eqref{mm15} are fulfilled e.g.~if
\begin{equation*}
a = \left( \frac{2\left( r^2+1 \right)^{\frac{2-\sigma}{2}}}{[(N-1) \wedge C^{\prime\prime}] \, \sigma m} \right)^{\frac 1{m-1}} \, .
\end{equation*}
In the case $ \gamma \le -2 $ (where $ \sigma=2 $) it is enough to exploit the validity of \eqref{e5}:
to make sure that $W$ satisfies \eqref{mm13}, it is easy to check that any $ r>0 $ and $a^{1-m}=(N-1)m$ will do.

Hence, we have provided a regular, positive function that satisfies \eqref{mm13} and belongs by construction to $ X_{\infty,\sigma} $. An immediate computation shows that the function $ V_T := T^{-1/(m-1)} \, W $ has the same properties and complies with \eqref{subsol-ellliptic}.
\end{proof}

\begin{proof}[Proof of Theorem \ref{opt1}]
For any $T>0$, let $V_T$ be the subsolution provided by Lemma \ref{lem: exist-subsol}. Given any $ \delta>0 $, let
\begin{equation}\label{choice-pre-pre}
V_{T, \delta} := \left(V_{T}^m - \delta\right)^{\frac 1 m} \vee 0 \, .
\end{equation}
It is not difficult to check that $ V_{T,\delta} $ is still a nontrivial function satisfying weakly \eqref{subsol-ellliptic}:
\begin{equation}\label{choice-pre}
V_{T,\delta} \le (m-1) T \, \Delta\!\left(V_{T,\delta}^m\right) \quad \textrm{in } M \, .
\end{equation}
Now we set
\begin{equation}\label{choice-T}
T=(2a)^{m-1}  \left[ \liminf_{\rho(x)\to\infty} \frac{u_0(x)}{\rho(x)^{\frac{\sigma}{m-1}}} \right]^{1-m} ,
\end{equation}
so that
\begin{equation}\label{choice-T-2}
\lim_{\rho(x)\to\infty} \frac{V_{T,\delta}(x)}{\rho(x)^{\frac{\sigma}{m-1}}} = \frac12 \, \liminf_{\rho(x)\to\infty} \frac{u_0(x)}{\rho(x)^{\frac{\sigma}{m-1}}}  \, .
\end{equation}
Because $ u_0 \ge 0 $, due to \eqref{choice-T-2} it is plain that for a suitable $ \delta $ large enough there holds $ u_0 \ge V_{T,\delta} $. Hence, thanks to Lemma \ref{lem: exist-max}, we know that the maximal existence time for the corresponding solution $u$ (which does exist and is nonnegative by Theorems \ref{thmexi}--\ref{thmuniq} since $ u_0 \in X_{\infty,\sigma} $) is at most $T$. From \eqref{choice-T} we then have that \eqref{mm30-datum} holds with $ \overline{C}=(2a)^{m-1} $. Finally, the validity of \eqref{mm30} is still a consequence of Theorem \ref{thmexi}: if $ \|u(t)\|_{\infty,r} $ stayed bounded up to $ t=T $ then we could extend the existence time for $u$ beyond $ T $, which is in contradiction with the maximality of $ T $.
\end{proof}

\begin{proof}[Proof of Corollary \ref{opt2}]
As a consequence of the method of proof of Theorem \ref{opt1} and in view of \eqref{mm22}, it is apparent that for any $T>0$ we can pick $ \delta $ so large that $ u_0 \ge V_{T,\delta} $. As $ V_{T,\delta} $ satisfies \eqref{choice-pre}, Lemma \ref{lem: exist-max} ensures that any nonnegative solution of problem \eqref{e64}, in the sense of Definition \ref{defsol}, with initial datum $u_0$ exists at most up to $ t=T $. Since $T$ can be arbitrarily small, the assertion follows.
\end{proof}

\begin{oss}\label{oss: local-neg}\rm
For simplicity, we stated and proved Theorem \ref{opt1} and Corollary \ref{opt2} for nonnegative data only. However, with minor modifications, they can be shown to hold also for initial data in $ L^\infty_{\rm loc}(M) $, provided \eqref{tmax-opt} or \eqref{mm22} is satisfied. In fact it is enough to replace the boundary condition in \eqref{a-mm26} with
$$ \hat{u}_{R} \,=\, \underset{x \in M}{\operatorname{essinf}} \, {u_0(x)} \quad \textrm{on } \partial B_R \times (0, \infty )  $$
and \eqref{choice-pre-pre} with
$$ V_{T, \delta} := \left(V_{T}^m - \delta\right)^{\frac 1 m}  \, . $$
Accordingly, the statement of Corollary \ref{opt2} should be modified by asserting that no solution $u$ larger than or equal to $ {\operatorname{essinf}}_{x \in M} \, {u_0(x)} $ (rather than nonnegative) exists.
\end{oss}

\smallskip

We now turn to pointwise blow-up. Before proving Theorem \ref{opt-blow}, we need a crucial lemma concerning the Cauchy problem \eqref{pb: cauchy}.

\begin{lem}\label{elliptic}
Let $ M \equiv M_{\psi} $ be any model manifold satisfying hypothesis \eqref{H} for some $ \gamma \in (-\infty,2) $. If $ \gamma \in (-2,2) $, assume in addition that \eqref{assumption-sectional} holds for some $ C_1 , R_1 > 0 $. Let $ T, \alpha>0 $. Then there exists a unique solution $ \mathsf{W}_{T,\alpha} $ to the Cauchy problem \eqref{pb: cauchy}, which is positive, belongs to $ X_{\infty,\sigma} $ and satisfies
\begin{equation}\label{sol-ellliptic-beh}
\frac{k_0}{T^{\frac{1}{m-1}}} \le \liminf_{\rho\to\infty} \frac{\mathsf{W}_{T,\alpha}(\rho)}{\rho^{\frac{\sigma}{m-1}}} \le \limsup_{\rho\to\infty} \frac{\mathsf{W}_{T,\alpha}(\rho)}{\rho^{\frac{\sigma}{m-1}}} \le \frac{k_1}{T^{\frac{1}{m-1}}}
\end{equation}
for positive constants $k_0$ and $k_1$ depending on $ C_0,\gamma,N,m $ and on $ C_1, R_1,\gamma,N,m $, respectively, but not on $\alpha$. Moreover, such solutions are strictly ordered with respect to $ \alpha $, namely $  \mathsf{W}_{T,\alpha_1} > \mathsf{W}_{T,\alpha_0} $ for all $ \alpha_1 > \alpha_0>0 $.
\end{lem}
\begin{proof}
We shall consider the case $ T=1/(m-1) $ only, and set $ \mathsf{W}:=\mathsf{W}_{1/(m-1),\alpha} $. The conclusions for general $T$ will follow by a scaling argument. Hence, let $ \mathsf{V}:= \mathsf{W}^{m} $. The differential equation in \eqref{pb: cauchy} can be rewritten as
\begin{equation}\label{sol-ellliptic-radial}
\left(\psi^{N-1} \, \mathsf{V}^\prime \right)^\prime = \psi^{N-1} \, \mathsf{V}^{\frac{1}{m}} \quad \textrm{in } (0,\infty) \, ;
\end{equation}
accordingly, the initial conditions read
\begin{equation}\label{sol-ellliptic-radial-init}
\mathsf{V}(0)=\alpha^m \, , \quad \mathsf{V}^\prime(0)=0 \, .
\end{equation}
Because $ m>1 $, $ \alpha>0 $, $ \psi(\rho) \sim \rho $ as $ \rho \to 0 $ and $ \psi^\prime(0)=1 $, standard fixed point arguments ensure that at least a local solution to \eqref{sol-ellliptic-radial}--\eqref{sol-ellliptic-radial-init} exists. On the other hand the r.h.s.~of \eqref{sol-ellliptic-radial} is sublinear in $ \mathsf{V} $ since $ m>1 $, whence the solution is in fact global in $ (0,\infty) $ (see e.g.~\cite[Theorem 1.6]{PS}).

From \eqref{sol-ellliptic-radial} it follows that $ \psi^{N-1} \, \mathsf{V}^\prime $ is increasing as long as $ \mathsf{V} $ stays positive, which implies that $ \mathsf{V} $ is actually positive and increasing everywhere thanks to \eqref{sol-ellliptic-radial-init}. We shall use this information in order to prove \eqref{sol-ellliptic-beh} (in this regard note that the estimate from above directly yields $ \mathsf{W} \in X_{\infty,\sigma} $). Indeed, by integrating \eqref{sol-ellliptic-radial} and using \eqref{sol-ellliptic-radial-init}, we get:
\begin{equation}\label{sol-ellliptic-radial-integ}
\mathsf{V}^\prime(\rho) = \frac{\int_{0}^{\rho} \psi(s)^{N-1} \, \mathsf{V}(s)^{\frac{1}{m}} \, ds }{\psi(\rho)^{N-1}}  \quad \forall \rho \in (0,\infty) \, .
\end{equation}
Let us assume for the moment that $ \psi $ complies with
\begin{equation}\label{psi-gen}
\lim_{\rho \to \infty}  \rho^{\sigma-1} \, \frac{\psi^\prime(\rho)}{\psi(\rho)} =: C>0
\end{equation}
for $ \sigma \in (0,2) $; we shall explain in the end of the proof how one can get rid of such extra assumptions. So, by exploiting the fact that $ \mathsf{V} $ is increasing, from \eqref{sol-ellliptic-radial-integ} we deduce
\begin{equation}\label{sol-ellliptic-radial-integ-bis}
\mathsf{V}(\rho)^{-\frac{1}{m}} \, \mathsf{V}^\prime(\rho) \le \frac{\int_{0}^{\rho} \psi(s)^{N-1} \, ds}{\psi(\rho)^{N-1}}  \quad \forall \rho \in (0,\infty) \, ,
\end{equation}
that is
\begin{equation*}\label{sol-ellliptic-radial-integ-ter}
\left( \mathsf{V}^{\frac{m-1}{m}} \right)^{\prime}\!(\rho) \le \frac{m-1}{m} \, \frac{\int_{0}^{\rho} \psi(s)^{N-1} \, ds}{\psi(\rho)^{N-1}}  \quad \forall \rho \in (0,\infty) \, .
\end{equation*}
By means of L'H\^{o}pital's rule, thanks to \eqref{psi-gen} one shows that
\begin{equation}\label{sol-ellliptic-radial-integ-quater}
\lim_{\rho\to\infty} \rho^{1-\sigma} \, \frac{\int_{0}^{\rho} \psi(s)^{N-1} \, ds}{\psi(\rho)^{N-1}} = \frac{1}{(N-1)C} \, .
\end{equation}
As a consequence,
\begin{equation}\label{sol-ellliptic-radial-integ-qqs}
\left( \mathsf{V}^{\frac{m-1}{m}} \right)^{\prime}\!(\rho) \le \hat{C} \, \rho \left(1+\rho\right)^{\sigma-2}  \quad \forall \rho \in (0,\infty)
\end{equation}
for another $ \hat{C}>0 $ depending on $ \psi , C , \sigma , N , m$. An integration of \eqref{sol-ellliptic-radial-integ-qqs} readily yields the last inequality in \eqref{sol-ellliptic-beh}.

We now aim at proving the first inequality in \eqref{sol-ellliptic-beh}. To this end, in addition to \eqref{psi-gen} we shall also assume that
\begin{equation}\label{psi-half}
\lim_{\rho \to \infty} \frac{\psi( \rho / {2})}{\psi(\rho)} = 0  \, .
\end{equation}
We shall show below why this is no loss of generality. We therefore proceed by means of a recursive procedure: namely, given $ n \in \mathbb{N} $, suppose that $ \mathsf{V} $ complies with
\begin{equation}\label{V-liminf-1}
\mathsf{V}(\rho) \ge c_n \, \rho^{\beta_n} \quad \forall \rho \in (1,\infty)
\end{equation}
for some $ \beta_n \ge 0$ and $c_n>0$. By plugging estimate \eqref{V-liminf-1} into \eqref{sol-ellliptic-radial-integ} and using the fact that $ \mathsf{V} $ is increasing, we obtain:
\begin{equation}\label{V-liminf-3}
\mathsf{V}^\prime(\rho) \ge \frac{\int_{\frac{\rho}{2}}^{\rho} \psi(s)^{N-1} \, \mathsf{V}(s)^{\frac{1}{m}} \, ds }{\psi(\rho)^{N-1}} \ge \frac{ c_n^{\frac1m} }{2^{\frac{\beta_n}{m}}} \, \rho^{\frac{\beta_n}{m}} \, \frac{\int_{\frac{\rho}{2}}^{\rho} \psi(s)^{N-1} \, ds}{\psi(\rho)^{N-1}} \quad \forall \rho \in (2,\infty) \, .
\end{equation}
By  \eqref{psi-gen}, \eqref{psi-half} and L'H\^{o}pital's rule, one checks that
\begin{equation}\label{V-rec-3}
\lim_{\rho\to\infty} \rho^{1-\sigma} \, \frac{\int_{\frac{\rho}{2}}^{\rho} \psi(s)^{N-1} \, ds}{\psi(\rho)^{N-1}} = \frac{1}{(N-1)C} \, ;
\end{equation}
in particular, we can assert that
\begin{equation}\label{V-rec-4}
\frac{\int_{\frac{\rho}{2}}^{\rho} \psi(s)^{N-1} \, ds}{\psi(\rho)^{N-1}} \ge \tilde{C} \, \rho^{\sigma-1} \quad \forall \rho \in ( 1,\infty)
\end{equation}
for another $ \tilde{C}>0 $ depending on $ \psi , C , \sigma , N , m$. By combining \eqref{V-liminf-3} and \eqref{V-rec-4} we get
\begin{equation}\label{V-liminf-4}
\mathsf{V}^\prime(\rho) \ge \frac{\tilde{C} \, c_n^{\frac{1}{m}}}{2^{\frac{\beta_n}{m}}} \, \rho^{\frac{\beta_n}{m}+\sigma-1}  \quad \forall \rho \in (2,\infty) \, .
\end{equation}
An integration of \eqref{V-liminf-4} from $ 2 $ to $ \rho $ yields
\begin{equation}\label{V-liminf-5}
\mathsf{V}(\rho) \ge \mathsf{V}(2) + \frac{\tilde{C} \, c_n^{\frac{1}{m}}}{2^{\frac{\beta_n}{m}} \left( \frac{\beta_n}{m}+\sigma \right)} \left[ \rho^{\frac{\beta_n}{m}+\sigma} - 2^{\frac{\beta_n}{m}+\sigma} \right] \quad \forall \rho \in (2,\infty) \, .
\end{equation}
Since $ \mathsf{V} $ is increasing, we know in particular that $ \mathsf{V}(\rho) \ge \alpha^m $ for all $ \rho \in (1,2] $. Using this information, it is not difficult to check that \eqref{V-liminf-5} implies e.g.
\begin{equation*}\label{V-liminf-6}
\mathsf{V}(\rho) \ge \left[ \frac{\tilde{C} \, c_n^{\frac{1}{m}}}{2^{\frac{\beta_n}{m}+1}\left( \frac{\beta_n}{m}+\sigma \right)} \wedge \frac{\alpha^m}{2^{\frac{\beta_n}{m}+\sigma}} \right] \rho^{\frac{\beta_n}{m}+\sigma} \quad \forall \rho \in (1,\infty) \, .
\end{equation*}
Summing up, starting from \eqref{V-liminf-1} we have deduced the validity of
\begin{equation*}\label{V-rec-next}
\mathsf{V}(\rho) \ge c_{n+1} \, \rho^{\beta_{n+1}} \quad \forall \rho \in (1,\infty)
\end{equation*}
with
\begin{equation}\label{V-rec-parameters}
\beta_{n+1}=\frac{\beta_n}{m} + \sigma \, , \quad c_{n+1} = \frac{\tilde{C} \, c_n^{\frac{1}{m}}}{2^{\frac{\beta_n}{m}+1}\left( \frac{\beta_n}{m}+\sigma \right)} \wedge \frac{\alpha^m}{2^{\frac{\beta_n}{m}+\sigma}} \, .
\end{equation}
As remarked above, we know that \eqref{V-liminf-1} is satisfied with $ \beta_0=0 $ and $ c_0=\alpha^m $.
If we start the recursive procedure with such data, it is direct to see that \eqref{V-rec-parameters} yields
\begin{equation*}\label{V-rec-limit}
\lim_{n\to\infty} \beta_n = \frac{\sigma m}{m-1} \, , \quad  \liminf_{n\to\infty} c_n = \underline{c}(\alpha)>0 \, ,
\end{equation*}
where the constant $ \underline{c} $ depends on $ \alpha,\tilde{C},\sigma,m $ (we emphasize its dependence on $ \alpha $ for later purpose). Upon passing to the limit in \eqref{V-liminf-1} as $ n \to \infty $ we get
\begin{equation}\label{V-rec-final}
\mathsf{V}(\rho) \ge \underline{c}(\alpha) \, \rho^{\frac{\sigma m}{m-1}} \quad \forall \rho \in (1,\infty) \, ,
\end{equation}
which entails
\begin{equation}\label{V-rec-final-bis}
\liminf_{\rho \to \infty} \frac{\mathsf{V}(\rho)}{\rho^{\frac{\sigma m}{m-1}}} \ge \underline{c}(\alpha) \, .
\end{equation}
In order to establish the first inequality in \eqref{sol-ellliptic-beh}, we still need to get rid of the dependence of $ \underline{c} $ on $ \alpha $. To this aim, let us carry out another iterative scheme. From \eqref{V-rec-final-bis} we deduce that, for any $ \varepsilon \in (0,\underline{c}) $, there holds
\begin{equation}\label{V-iter-bis-1}
\mathsf{V}(\rho) \ge \left( \underline{c}-\varepsilon \right) \rho^{\frac{\sigma m}{m-1}}
\end{equation}
for all $ \rho $ large enough. By plugging \eqref{V-iter-bis-1} into \eqref{sol-ellliptic-radial-integ} and reasoning as above, we get
\begin{equation}\label{V-iter-bis-2}
\mathsf{V}^\prime(\rho) \ge \frac{\tilde{C} \left( \underline{c}-\varepsilon \right)^{\frac1m}}{2^{\frac{\sigma}{m-1}}} \, \rho^{\frac{\sigma m}{m-1}-1}
\end{equation}
for all $ \rho $ large enough. Upon integrating \eqref{V-iter-bis-2}, it follows that:
\begin{equation*}\label{V-iter-bis-3}
\liminf_{\rho \to \infty} \frac{\mathsf{V}(\rho)}{\rho^{\frac{\sigma m}{m-1}}} \ge \frac{(m-1) \, \tilde{C} \left( \underline{c}-\varepsilon \right)^{\frac1m}}{2^{\frac{\sigma}{m-1}} \, \sigma m } \, ,
\end{equation*}
whence
\begin{equation}\label{V-iter-bis-4}
\liminf_{\rho \to \infty} \frac{\mathsf{V}(\rho)}{\rho^{\frac{\sigma m}{m-1}}} \ge \frac{(m-1) \, \tilde{C} \, \underline{c}^{\frac1m}}{2^{\frac{\sigma}{m-1}} \, \sigma m }
\end{equation}
as $ \varepsilon $ is arbitrarily small. So from \eqref{V-rec-final-bis} we have deduced \eqref{V-iter-bis-4}: since $ m>1 $, one sees that by proceeding iteratively we end up with
\begin{equation}\label{V-iter-bis-5}
\liminf_{\rho \to \infty} \frac{\mathsf{V}(\rho)}{\rho^{\frac{\sigma m}{m-1}}} \ge \left[ \frac{(m-1) \, \tilde{C}}{2^{\frac{\sigma}{m-1}} \sigma m } \right]^{\frac{m}{m-1}} .
\end{equation}

\smallskip
Let us now comment on the fact that it is not restrictive to assume \eqref{psi-gen} and \eqref{psi-half}. In view of \eqref{H}-(i) and \eqref{assumption-sectional}, by proceeding as outlined in the proof of Lemma \ref{prop-comp-sect}, it is indeed possible to construct a suitable function $ {\psi}_\ast \in \mathcal{A} $ such that
\begin{equation}\label{psi-ast-1}
 \frac{\psi_\ast^\prime}{\psi_\ast} \le \frac{\psi^\prime}{\psi} \, , \quad  \lim_{\rho \to \infty}  \rho^{\sigma-1} \, \frac{\psi_\ast^\prime(\rho)}{\psi_\ast(\rho)} =: C_\ast>0 \, ,
\end{equation}
for some constant $C_\ast$ depending only on $C_1,R_1,\gamma,N,m$. In fact $ \psi_\ast $ is not explicit, but it is chosen as the solution of an explicit second-order linear ODE: we refer again the reader to \cite[Section 4]{GMV}. Thanks to \eqref{psi-ast-1} it is immediate to see that, because $ \mathsf{V}$ is increasing, the computations that led to \eqref{sol-ellliptic-radial-integ-qqs} can be repeated starting from \eqref{sol-ellliptic-radial-integ-bis} with $ \psi $ replaced by $ \psi_\ast $. Moreover, since we are supposing that $ \sigma \in (0,2) $, one checks that the right-hand equality in \eqref{psi-ast-1} also implies \eqref{psi-half} with $ \psi $ replaced by $ \psi_\ast $. Similarly, in view of \eqref{H}-(ii), by arguing as explained in the proof of Lemma \ref{prop-comp-ricci} we can infer the existence of another function $ \psi_\ast \in \mathcal{A} $ satisfying \eqref{psi-ast-1} with the left-hand inequality reversed: this, combined with \eqref{psi-half}, is enough in order to reproduce the computations that led us to \eqref{V-rec-final} and \eqref{V-iter-bis-5}.

In the cases $  \gamma \le -2 $, which correspond to $ \sigma=2 $, the functions $ \psi_\ast $ one constructs behave like powers at infinity (see \cite[Sections 8.1--8.2]{GMV}); as concerns the estimate from above \eqref{sol-ellliptic-radial-integ-qqs}, one can even choose $ \psi_\ast(\rho)=\rho $. This is enough to establish the analogues of \eqref{sol-ellliptic-radial-integ-quater} and \eqref{V-rec-3}, the latter being the core of the above arguments.
%
%

\smallskip
Finally, let us show that solutions are strictly ordered with respect to $ \alpha $. Given $ \alpha_1>\alpha_0>0 $, one has indeed
\begin{equation*}\label{diff-alpha}
(m-1)T\left[\psi^{N-1} \left( \mathsf{W}_{T,\alpha_1}^m - \mathsf{W}_{T,\alpha_0}^m \right)^\prime \right]^\prime = \psi^{N-1} \left( \mathsf{W}_{T,\alpha_1}-\mathsf{W}_{T,\alpha_0} \right) \quad \textrm{in } (0,\infty) \, .
\end{equation*}
Since $ \mathsf{W}_{T,\alpha_1}(0)-\mathsf{W}_{T,\alpha_0}(0)=\alpha_1-\alpha_0>0 $ and $ ( \mathsf{W}_{T,\alpha_1}^m - \mathsf{W}_{T,\alpha_0}^m )^\prime(0)=0 $ we deduce that, away from the origin, $ ( \mathsf{W}_{T,\alpha_1}^m - \mathsf{W}_{T,\alpha_0}^m )^\prime $ stays positive as long as $ \mathsf{W}_{T,\alpha_1}-\mathsf{W}_{T,\alpha_0} $ is, which easily yields positivity everywhere.
\end{proof}

\begin{proof}[Proof of Theorem \ref{opt-blow}]
First of all note that, in view of \eqref{mm43} and Lemma \ref{elliptic}, any $ \mathsf{W}_{T,\alpha} \equiv \mathsf{W}_{T,\alpha}(\rho(x)) $ is a positive solution, belonging to $ X_{\infty,\sigma} $, of the elliptic equation
\begin{equation*}\label{eq: diff-W}
\mathsf{W}_{T,\alpha} = (m-1)T \, \Delta\!\left( \mathsf{W}_{T,\alpha}^m \right) \quad \textrm{in } M_\psi \, .
\end{equation*}
By arguing as in the proof of Proposition \ref{supersol}, we then infer that the separable profile
$$
(x,t) \mapsto \left( 1 - \frac{t}{T} \right)^{-\frac{1}{m-1}} \mathsf{W}_{T,\alpha}(\rho(x)) \quad \forall (x,t) \in M_\psi \times (0,T)
$$
is the solution of \eqref{e64} with initial datum $ \mathsf{W}_{T,\alpha} $. By\eqref{W-W} and using Corollary \ref{comppr}, we obtain
\begin{equation}\label{eq: diff-W-1}
\left( 1 - \frac{t}{T} \right)^{-\frac{1}{m-1}} \mathsf{W}_{T,\alpha_0}(\rho(x)) \le u(x,t) \le \left( 1 - \frac{t}{T} \right)^{-\frac{1}{m-1}} \mathsf{W}_{T,\alpha_1}(\rho(x)) \quad \forall (x,t) \in M_\psi \times (0,T)  \, .
\end{equation}
From \eqref{eq: diff-W-1} it is apparent that the maximal existence time for $u$ is precisely $T$ and that \eqref{blow-up-ptws} holds, whereas the validity of \eqref{eq: u0-blowup} is a direct consequence of \eqref{W-W} and \eqref{sol-ellliptic-beh}.
\end{proof}


\end{document}